\documentclass[notitlepage]{revtex4-1}
\usepackage{amsthm}
\usepackage{lmodern}
\usepackage[english]{babel}
\usepackage{amsmath}
\usepackage{amssymb}
\usepackage{mathptmx} 
\usepackage{color}
\usepackage[babel]{microtype}
\usepackage{graphicx}
\usepackage{caption}
\usepackage{subcaption}
\usepackage{csquotes}
\usepackage{float}
\usepackage[all,cmtip]{xy}
\usepackage{bm}
\usepackage[a4paper, left=3cm, right=2cm]{geometry}
\usepackage{enumitem}
\newtheorem{theorem}{\bf Theorem}[section]

\newcommand{\rmi}{\mathrm{i}}

\newcommand{\sign}{\operatorname{sign}}

\begin{document}

\title{Knotted fields and explicit fibrations for lemniscate knots}

\author{
B.~Bode\footnote{benjamin.bode@bristol.ac.uk}, M.~R.~Dennis\footnote{mark.dennis@bristol.ac.uk}, D.~Foster and R.~P.~King}

\address{H H Wills Physics Laboratory, University of Bristol, Bristol BS8 1TL, UK
}




\begin{abstract}
We give an explicit construction of complex maps whose nodal line have the form of lemniscate knots. We review the properties of lemniscate knots, defined as closures of braids where all strands follow the same transverse (1, $\ell$) Lissajous figure, and are therefore a subfamily of spiral knots generalising the torus knots. 
We then prove that such maps exist and are in fact fibrations with appropriate choices of parameters. We describe how this may be useful in physics for creating knotted fields, in quantum mechanics, optics and generalising to rational maps with application to the Skyrme-Faddeev model. We also prove how this construction extends to maps with weakly isolated singularities. 
\end{abstract}
\maketitle

\section{Introduction}\label{sec:intro}

A challenging problem in knot theory is writing down an explicit parametrisation of a curve in the form of a given type of knot or link.
This problem becomes even more demanding when we ask for an explicit function of three-dimensional space which contains special loci which are knotted and linked, such as a complex-valued function $\mathbb{R}^{3}\longrightarrow\mathbb{C}$ with a zero level set of the form of a given knot, a model \emph{knotted field}. 

The study and construction of knotted fields of this type is of interest in both mathematics and physics. 
In particular, various kinds of knotted fields find application in areas such as topological fluid dynamics \cite{moffatt:1969degree,lr:2012jones}, nonlinear field theories \cite{sutcliffe:2007knots}, nematic liquid crystals \cite{ma:2014knotted}, excitable media \cite{winfree:1994persistent}, electromagnetic fields \cite{ib:2008linked} and optical physics \cite{bd:2001knotted,dkjop:2010isolated}. 
The physicist's interest is then usually in the specific conformation of the knot when the function minimises some energy functional or a solution of some linear or nonlinear PDE.

Here, we show how to generate explicit, complex scalar knotted fields for the family of knots we call \emph{lemniscate knots}, which, in a generalisation of the procedure described in \cite{dkjop:2010isolated}, are built up from explicit constructions of certain braids, which are defined by trigonometric functions.
Lemniscate knots have not, to our knowledge, been emphasised as a class within knot theory---they are a subclass of so-called \emph{spiral knots} \cite{betvwy:2010spiral}---and are automatically fibred (i.e.~the complement of the knot can be divided into topologically equivalent surfaces parametrised by points on a circle); we prove that the lemniscate knotted fields arising from the construction can be explicit fibrations.
Theorems related to the Nash-Tognoli Theorem \cite{bcr:1998real} imply that every knot or link type can be realised as the intersection of the zero sets of two real polynomials in the spatial variables $x$, $y$ and $z$.
However, such theorems are rarely constructive, leaving a physicist or applied mathematician, wishing for an explicit analytic function representation for a knotted field, at a loss on how to proceed.

A construction by Brauner in 1928 \cite{brauner:1928geometrie} gives an explicit realisation of such maps for the $(p,q)$-torus knot or link. 
This begins by constructing a complex polynomial of two complex variables $u$ and $v$, 
\begin{equation}
   f_{\mathrm{Brauner}}:\mathbb{C}^{2}\longrightarrow\mathbb{C}, \qquad f_{\mathrm{Brauner}}(u,v)=u^{p} - v^{q}.
   \label{eq:brauner}
\end{equation}
With the restriction $|u^2|+|v^2| = 1$, $(u,v)$ can be used as complex coordinates for the unit three-sphere $S^3$.
Stereographically projecting to $\mathbb{R}^3$, for example with the explicit choice of projection 
\begin{equation}
   u = \frac{r^2 -1 + 2 \rmi z}{r^2 + 1}, \qquad v = \frac{2(x+\rmi y)}{r^2 + 1},
   \label{eq:coords}
\end{equation}
where $\boldsymbol{r} = (x,y,z)$, makes $f_{\mathrm{Brauner}}$ in (\ref{eq:brauner}) an explicit complex rational function of $x$, $y$ and $z$, which indeed has a nodal line in the form of the $(p,q)$ torus knot (if $p$ and $q$ are coprime, otherwise it is a torus link). 
If necessary only the numerator (a polynomial in $x,y,z$) of the rational function $f_{\mathrm{Brauner}}$ can be considered; in both cases the zero level set contains the desired knot.

Brauner's method can be understood in terms of closing braids to obtain knots or links.  
The $(p,q)$ torus knot or link is the closure of the simple braid where with $p$ strands forming a helix, undergoing $q/p$ full twists, as the example in Figure \ref{fig:torusknot}. 
This braid maps to a 1-parameter family of complex polynomials, parametrised by real $h$ acting as braid height, and the $s$ roots sweep out the $s$ strands of the braid as $h$ increases.
This leads directly to the function $u^p - v^q$, where the level set of zero forms the $(p,q)$ torus link. 
This can be seen by mapping the complex braid into the complex $u,v$ coordinates of the 3-sphere, which closes the braid.

\begin{figure}
\begin{center}
\includegraphics[height=4.5cm]{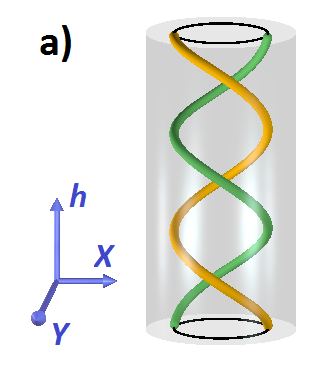}\hspace{2cm}
\includegraphics[height=4.5cm]{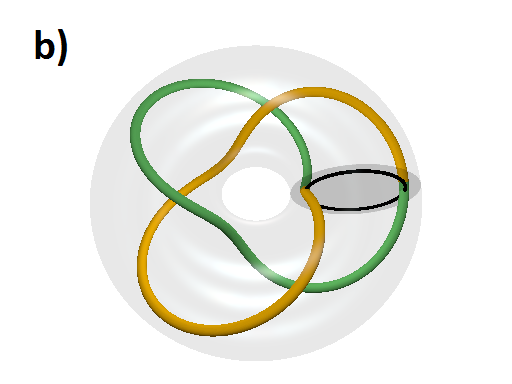}
\caption{
   Closing a helical braid on a cylinder to a torus knot. 
   (a) The strands of the 2-strand braid $\sigma^{3}$ form a three-crossing double helix inside a cylinder. $X$, $Y$ and $h$ increase in the direction of the corresponding arrow.
   (b) Gluing the top and bottom of the cylinder together closes the braid, yielding a trefoil knot inside a torus.
   }
\label{fig:torusknot}
\end{center}
\end{figure}

This approach led Perron to construct a polynomial map $f:\mathbb{R}^{4}\longrightarrow\mathbb{R}^{2}$ for the figure-8 knot $4_1$, the simplest non-torus knot \cite{perron:1982noeud}, by a similar construction involving the pigtail braid shown in Figure \ref{fig:figure8}.
Although $4_1$ cannot be represented by a braid that can be drawn on a cylinder, it can be arranged to lie on the surface of the cartesian product of a lemniscate with an interval.
This leads naturally to a trigonometric parametrisation of the braid as the lemniscate of Gerono (a $(1,2)$-Lissajous figure), and the procedure then essentially follows Brauner's construction.
The braid in Figure \ref{fig:figure8} consists of three strands with different starting points following the same lemniscate path along the braid (as height $h$ increases).
Two vertical periods of the characteristic alternating crossing pattern of the pigtail braid are shown.
Perron's construction generalises to all knots and links which can be formed as closures of the braid which consists of $r$ copies of this same basic braid ($r = 2$ is $4_1$, $r = 3$ gives the borromean rings $L6a4$, $r=4$ the knot $8_{18}$, etc).
Similar functions based on this lemniscate were constructed by Rudolph \cite{rudolph:1987isolated} and later Dennis et al.~\cite{dkjop:2010isolated}, of the form $f:\mathbb{C}\times\mathbb{R}^{2}\longrightarrow\mathbb{C}$, which are holomorphic in one variable (i.e.~complex analytic in $u$ but not $v$), which we call \emph{semiholomorphic}.

\begin{figure}
\begin{center}
\begin{subfigure}[t]{0.4\textwidth}
\centering
\includegraphics[height=4.5cm]{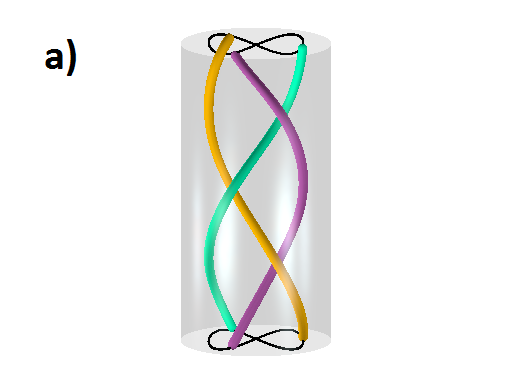}
\end{subfigure}
\begin{subfigure}[t]{0.4\textwidth}
\centering
\includegraphics[height=4.5cm]{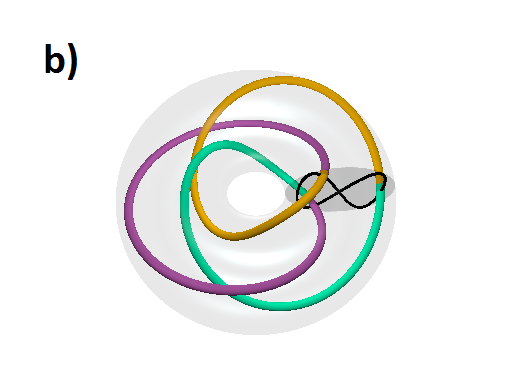}
\end{subfigure}
\caption{The lemniscate braid with braid word $(\sigma_{1}^{-1}\sigma_{2})^{2}$, closing to the figure-8 knot. 
(a) The strands of the braid $(\sigma_1^{-1}\sigma_2)^{2}$ drawn in a cylinder. 
The strands are parametrised such that they lie on the `$\infty$' figure times $[0,2\pi]$ (the pigtail braid).
(b) Gluing the the ends of the cylinder yields a torus enclosing the figure-8 knot.}
\label{fig:figure8}
\end{center}
\end{figure}

Here, we generalise the construction further, explicitly constructing semiholomorphic maps with nodal sets in the form of families of knots based on Lissajous figures.
We call these families \emph{lemniscate knots}, which are based on braids whose strands follow the same generalised lemniscate $(1,\ell)$ Lissajous figure (a true lemniscate has $\ell = 2$), and are trigonometric functions of braid height parameter $h$.
The resulting semiholomorphic complex function $f(u,v,\overline{v})$, with lemniscate knotted nodal set, gives a complex scalar field $F(\boldsymbol{r})$ of 3-dimensional space by (\ref{eq:coords}), whose nodal lines (phase singularities) have the form of the lemniscate knot.
Such functions $F$ can then be used for various physical applications, such as holograms to create knotted optical vortices \cite{dkjop:2010isolated}, as templates for vector fields whose helicity is determined by the knot \cite{kfdi:2016weaving}, or candidate, knotted minimum-energy solutions for the Skyrme-Faddeev model \cite{sutcliffe:2007knots}.

The maps constructed by Brauner, Perron and Rudolph in fact satisfy much stronger properties than just having the correct topology on the unit three-sphere, defining neighbourhoods of singularities in 4-dimensional real space, as studied extensively by Milnor \cite{milnor:1968singular}. 
Although we have been unable to extend the mappings based on lemniscate knots to singularities of semiholomorphic type (as Rudolph), nevertheless the symmetries of the lemniscate braids often do allow the explicit construction of polynomial maps with weakly isolated singularities of the type originally described by Akbulut and King \cite{ak:1981all}.

As a basis for our later construction, it is helpful to formalise the procedure for torus knots and the figure-8 knot discussed above, generalising to arbitrary lemniscate braids.
In the horizontal plane transverse to the braid height, the $s$ strands follow the generalised lemniscate curve ($(1,\ell)$ Lissajous figure), parametrised by $h$, given by $(X^{s,r}_j(h),Y^{s,r,\ell}_j(h))$, where $j = 1,\dots,s$, and
\begin{equation}
   X^{s,r}_j(h) = a \cos\left(\frac{1}{s} [r h + 2\pi (j-1)]\right), \qquad 
   Y^{s,r,\ell}_j(h) = \frac{b}{\ell}\sin\left(\frac{\ell}{s} [r h + 2\pi (j-1)]\right).
   \label{eq:XYsrl}
\end{equation}
Here, $a$, $b\in\mathbb{R}$ are stretching factors, set to unity unless otherwise stated.
The prefactor $\ell^{-1}$ in $Y^{s,r,\ell}_j(h)$ ensures that each `lobe' in the Lissajous figure has aspect ratio approximately unity when $a = b$.
$X$ is simply a cosine function, and is independent of $\ell$.
The $j$th point (representing a strand of the braid) moves cyclically to the $j+1$th point ($j = 1, 2,\ldots, s$), and this pattern (equivalent algebraically to a basic braid word) repeats $r$ times as $0 \le h \le 2\pi$.  
In the 3-dimensional space of the braid, the strands follow the curve $\boldsymbol{S}^{s,r,\ell}_j(h)$ parametrised by the height coordinate $h$, with $0 \le h \le 2\pi$, and $h$ increasing upwards,
\begin{equation}
   \boldsymbol{S}^{s,r,\ell}_j(h) = \left(X^{s,r}_j(h), Y^{s,r,\ell}_j(h), h\right).
   \label{eq:strand}
\end{equation}
This braid is represented by the family of complex polynomials $p^{s,r,\ell}_h(u)$, with variable $u \in \mathbb{C}$ and real, cyclic parameter $h\in [0,2\pi]$, that have roots $Z^{s,r,\ell}_j(h)$ given by the intersection of the parametrised braid with the horizontal plane (now taken to be the complex plane) at height $h$, i.e.~$Z^{s,r,\ell}_j(h)= X^{s,r,\ell}_j(h) + \mathrm{i} Y^{s,r,\ell}_j(h)$, so
\begin{equation}
   p^{s,r,\ell}_h(u) = \prod_{j = 1}^s (u - Z^{s,r,\ell}_j(h)).
   \label{eq:polydef}
\end{equation}
The semiholomorphic map $f(u,v,\overline{v})$ with knotted zero line is found by the replacement, in $p^{s,r,\ell}$, of $\exp(\mathrm{i}h)$ with $v$ and $\exp(-\mathrm{i}h)$ with $\overline{v}$, ensuring $v\overline{v} = 1$.
This is proved later in Section \ref{sec:fibration}.

When $\ell = 1$, the braid is helical, closing to a $(s,r)$ torus knot or link.
Assuming $a=b$, each strand follows a circle of radius $a$ in the horizontal plane, and the $s$ strands are uniformly distributed around this circle. 
After a $2\pi$ increase of $h$, the regular $s$-gon of intersections of strands with the horizontal plane has turned by $2\pi r / s$. 
Since each root $Z^{s,r,1}_j$ has the form $a \exp( \mathrm{i} h r/s)$ times an $s$th root of unity, the polynomial (\ref{eq:polydef}) multiplies out to the form $p^{s,r,1}_h(u) = u^s - a \exp(\mathrm{i} h r)$.
In this case, the map $f$ arises from $p_h(u)$ on identifying $\exp(\mathrm{i} h) \to v$, giving, for the $(s,r)$ torus knot, $f = u^s - v^r$, equivalent to $f_{\mathrm{Brauner}}$ above.
The explicit knotted field of $\mathbb{R}^3$ arises from the substitution (\ref{eq:coords}).

The procedure of creating the figure-8 knot uses $\ell = 2$ (following the approach of \cite{dkjop:2010isolated}), and replaces the circular trajectory in the horizontal plane of the braid with the lemniscate $(\cos t,\tfrac{1}{2}\sin2t)$ (assuming $a = b =1$); the braid whose closure is the figure-8 knot has three strands and two repeats of the basic period, so the figure-8 knot has $(s,r,\ell) = (3,2,2)$ in (\ref{eq:XYsrl}), and the corresponding polynomial (\ref{eq:polydef}) has roots $Z^{3,2,2}_j(h) = \cos\left(\tfrac{1}{3}[2h+2\pi(j-1)]\right)+\tfrac{1}{2}\mathrm{i} \sin\left(\tfrac{2}{3}[2h+2\pi(j-1)]\right)$, and, after multiplying out the polynomial and making the identification $\exp(\mathrm{i} h) \to v$, $\exp(-\mathrm{i} h) \to \overline{v}$, we get the function \cite{dkjop:2010isolated}
\begin{equation}
   f_{\mathrm{fig}\mbox{-}8}(u,v,\overline{v}) = 64 u^3 - 12 u (3 + 2 [v^2 - \overline{v}^2] ) - 14(v^2 + \overline{v}^2) - (v^4 - \overline{v}^4).
   \label{eq:ffig8}
\end{equation}
Unlike the function $u^s - v^r$ for torus knots, this function is semiholomorphic, depending on both $v$ and $\overline{v}$, a consequence of the fact that the trigonometric functions in $Z^{3,2,2}_j(h)$ have different arguments; underlying this is the fact that only braids where all crossings are over-crossings (as in Figure \ref{fig:torusknot} but not Figure \ref{fig:figure8}) can be represented by fully holomorphic polynomials.
For closures of the braid with different choices of $r$, $v^2$ and $\overline{v}^2$ are replaced in (\ref{eq:ffig8}) with $v^r, \overline{v}^r$, giving, the borromean rings $r = 3$, etc.~as discussed above.
Equation (\ref{eq:ffig8}) gives an explicit figure-8 knotted field in $\mathbb{R}^3$ with the identification (\ref{eq:coords}).

The structure of this paper proceeds as follows.
In the next section, we consider the mathematical properties of lemniscate knots.
In Section \ref{sec:fibration}, we prove that the map $f$ described above, constructed from the polynomial (\ref{eq:polydef}), has the desired knot and link, for appropriate choices of $a$ and $b$. 
Furthermore, with appropriate $a$ and $b$, we prove the argument of the resulting function $\arg(f)$ gives a fibration of the knot complement over $S^{1}$.
Physical applications of the procedure involving knotted fields in quantum mechanics, optics and Skyrme-Faddeev hopfions are discussed in Section \ref{sec:applications}.
The construction of polynomial maps with weakly isolated singularities of the form of Akbulut and King is discussed in Section \ref{sec:AK}, before a concluding discussion in Section \ref{sec:discussion}.

\section{Lemniscate knots: braids and properties}\label{sec:braids}

Lemniscate knots and links are defined as the closures of braids whose strands execute the same generalised lemniscate trajectory in the horizontal plane, as given in (\ref{eq:XYsrl}), (\ref{eq:strand}).
They are determined by three positive integers: the number of strands $s$, the number of repeats $r$ of the basic pattern, and the number of lobes $\ell$ in the generalised lemniscate, i.e.~in the $(1,\ell)$ Lissajous figure.
We will often refer to the functions $X(h), Y(h)$, suppressing suffixes when the context is clear.
Different choices of $a$ and $b$ give the same braid which is rescaled in horizontal plane.
We always assume that $a>0$.
Replacing $b$ by $-b$ gives the braid which is the mirror image and hence the closures of the braids corresponding to $b$ and $-b$ are also mirror images. 
Unless stated otherwise, we assume $b > 0$. 
Otherwise $a$ and $b$ are just scale factors, and do not change the topology.

Equation (\ref{eq:XYsrl}) not only parametrises the braid, but also its closure. 
The lemniscate knot which is the closure of the braid in Equation \ref{eq:XYsrl} can be parametrised by 
\begin{equation}
	\left(\cos(h)(R+X_{1}^{s,r}(h)),\sin(h)(R+X_{1}^{s,r}(h)),Y_{1}^{s,r,\ell}(h)\right),\qquad
	h\in[0,2s\pi]
	\label{eq:para}
\end{equation}
where $R>0$ is large enough that $(R+X_{1}^{s,r}(h)) > 0$ for all $h\in[0,2s\pi]$ (and upwards increase of $h$ corresponds to right-handed increase of azimuthal coordinate in the solid torus). 
The idea behind this parametrisation can be understood as taking the parametrised braid inside a cylinder of radius $R$ and wrapping it around, joining top and bottom of the cylinder.
This is illustrated in Figures \ref{fig:torusknot} and \ref{fig:figure8}, and also Figure \ref{fig:ell3}, which involves an $\ell = 3$ generalised lemniscate, with five strands and two repeats, closing to the knot $7_7$ (as discussed later at Table \ref{tab:ell=3}).
Note that in this process we have identified the braid height coordinate $h$ with an azimuthal coordinate of the solid torus in which the knot is embedded. 
We will revisit this idea in the construction of knotted fields in later sections.

\begin{figure}
\begin{center}
\begin{subfigure}[t]{0.4\textwidth}
\centering
\includegraphics[height=4.5cm]{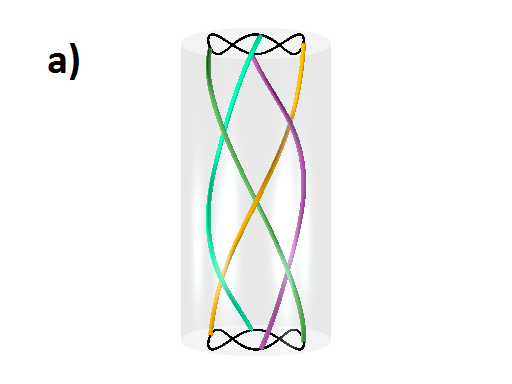}
\end{subfigure}
\begin{subfigure}[t]{0.4\textwidth}
\centering
\includegraphics[height=4.5cm]{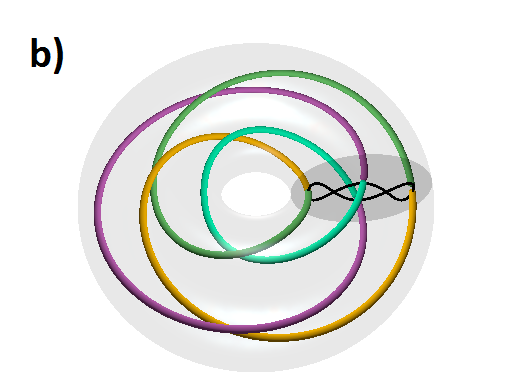}
\end{subfigure}
\caption{A braid with an $\ell=3$ horizontal lemniscate trajectory, and its closure. 
(a) The trigonometric braid with parameters $(s,r,\ell) = (4,2,3)$.
(b) Closure of the braid, isotopic to the link $L6a1$.}
\label{fig:ell3}
\end{center}
\end{figure}

From the point of view of constructing fields with specified zero lines (as for parametrisations of knotted curves) from braids via (\ref{eq:para}), it is most natural to specify the braid's strands as parametric curves; however, the usual mathematical framework of braids is via the algebra of crossings of the Artin braid group (described in \cite{kt:2008braid}), which we briefly review before examining the braids closing to lemniscate knots.
The group's generators are the crossings $\sigma_k$, $k = 1, \ldots, s-1$, with $k$ labelling the crossing position in order from the left, with positive power for an overcrossing, and the inverse $\sigma_k^{-1}$ for an undercrossing.
A product of generators $\sigma_{k_{1}}^{\varepsilon_{1}}\sigma_{k_{2}}^{\varepsilon_{2}}\ldots\sigma_{k_{n}}^{\varepsilon_n}$ is called a \emph{braid word}, and represents the geometric braid which is formed of strands that perform the crossings specified in the braid word from left to right following the braid downwards.
Isotopic braids are equivalent under the braid group relations: generators $\sigma_k$, $\sigma_{k'}$ commute unless $|k-k'| = 1$, in which case $\sigma_k \sigma_{k+1} \sigma_k = \sigma_{k+1} \sigma_k \sigma_{k+1}$. 
This algebraically represents the third Reidemeister move; the fact that $\sigma_k \sigma_k^{-1}$ is the identity represents the second Reidemeister move.

For knots and links formed by the closure of the braid represented by a word $w$, there are two additional \emph{Markov moves}: if the generator $\sigma_1$ with power $\pm 1$ (equivalently, $\sigma_{s-1}^{\pm 1}$) occurs exactly once in $w$, then the knot is isotopic to the braid closure of the word with $\sigma_1$ ($\sigma_{s-1}^{\pm 1}$) omitted (this stabilisation move is equivalent to the first Reidemeister move); the conjugation move states that the knot which is the closure of $w = w' \sigma_k$ is isotopic to the closure of $\sigma_k w'$.
A braid defines a permutation on the strand labels; the number of disjoint cycles of the permutation gives the number of disjoint components of the link upon closing the braid.

The convention in braid theory (e.g.~\cite{kt:2008braid}) is to read the braid word from a diagram \emph{downwards}, with a left-handed orientation (i.e.~overcrossings of strands from the left are considered positive).
The figures of the parametric curves \ref{fig:torusknot}, \ref{fig:figure8}, \ref{fig:ell3} are drawn in a left-handed coordinate system $(X,Y,h)$, where $h$ increases upwards, so that the figures show braid diagrams that allow to read off the braid word in the usual way.
We will attempt to respect both conventions of reading braid words downwards and describing the geometric movements of the strands in the direction of increasing $h$, upwards.
This convention means that under braid closure, $X,Y,h$ map, in cylindrical coordinates, to radius, height and azimuth respectively as in (\ref{eq:para}).

The diagram of the braid comprised of the $s$ strands following $(X_j(h),Y_j(h),h)$, $j = 1, \ldots,s$, is defined to be the projection of the braid in the $(0,1,0)$-direction, i.e.~the braid diagram is made up of the curves $(X_j(h),h)$, and crossing signs are determined by $Y_j(h)$. 
Crossings occur when $(X_{j}(h'),h')=(X_{j'}(h'),h')$ at some $h=h'$ and some $j,j'$ with $j\neq j'$.
Our convention is that the strand labelled by $j$ crosses over $j'$ at $h$ when $Y_{j}(h) > Y_{j'}(h)$.
From the form of $X^{s,r}_j(h)$ in (\ref{eq:XYsrl}), this requirement is satisfied if and only if $h r = 0\mod 2\pi$ or $h r = \pi\mod 2\pi$. 
We choose to place the crossings at $h=2\pi$ at the beginning of the braid word (identical, by Markov conjugation, to considering them at $h=0$ at the end of the word). 
Note that the crossings at $h r=0\mod 2\pi$, $\sigma_{k}^{\varepsilon_{k}}$ with $k$ odd, are simultaneous in $h$, as are the crossings at $h r=\pi\mod 2\pi$ with $k$ even.

For $r = 1$, this gives the basic braid word, corresponding to the parametrised strands (\ref{eq:XYsrl}), 
\begin{equation}
   w^{(s,r=1,\ell)} = \sigma_1^{\varepsilon_1}\sigma_3^{\varepsilon_3}\sigma_5^{\varepsilon_5}\cdots\sigma_2^{\varepsilon_2}\sigma_4^{\varepsilon_4}\sigma_6^{\varepsilon_6}\cdots,
   \label{eq:bbw}
\end{equation}
and for general $r$, we have $r$ repeats, i.e.~$w^{(s,r,\ell)} = (w^{(s,1,\ell)})^r$.
The signs of the crossings $\varepsilon_k$ are determined by $\ell$; the crossing sign depends on the parity of the lobe of the Lissajous figure in which the crossing occurs.
Since the strands cannot intersect, $s$ and $\ell$ must be relatively coprime (so there is no $h$ for which $(X_j(h),Y_j(h),h)=(X_{j'}(h),Y_{j'}(h),h)$ for some $j\neq j'$). 
We choose $s > \ell$, since for any knot with $s <\ell$, there is $\ell'<s$ which gives rise to the same braid word.
It is straightforward to see that the crossing signs, for $j = 1, \ldots, s-1$, and $s$ and $\ell$ coprime, are determined by the following rule:
\begin{equation}
	\varepsilon_{s-1}=\sign(b)\qquad \hbox{and }\varepsilon_{j}=\begin{cases} -\varepsilon_{j+1} \hbox{ if there is an integer $m$ with }\frac{j}{s}<\frac{m}{\ell}<\frac{j+1}{s},
	\\ +\varepsilon_{j+1} \hbox{ otherwise.}\end{cases}
   \label{eq:epscondition}
\end{equation}
For $\ell = 1$ and $b>0$, this implies that the crossings are all positive, as expected for braid representations of torus knots.
For $\ell = 2$, $s$ must be odd, and so all crossings $j < s/2$ are positive, and $j > s/2$ are negative if $b>0$.
Thus the braid representation of the figure-8 knot is $(\sigma_1^{-1} \sigma_2)^2$. 
Equation (\ref{eq:epscondition}) describes the signs of lemniscate braid with positive $b$. 
Note that for negative values of $b$ the vector $\boldsymbol{\varepsilon}=(\varepsilon_{1},\varepsilon_{2},\ldots,\varepsilon_{s-1})$ is exactly the negative of the $\boldsymbol{\varepsilon}$ for positive $b$.
In general, the lemniscate braid representation is not the minimal braid representation of the knot or link as found by \cite{gittings:2004minimum}.

The lemniscate knots thus described are in the more general family of \emph{spiral knots} (or links); a spiral knot (link) is defined as the closure of the $r$th power of a braid word $w$ in which every generator $\sigma_k$ appears exactly once, either as positive or negative power $\varepsilon_k = \pm 1$ \cite{betvwy:2010spiral}.
We term such braid words \emph{isograms}.
The braid word (\ref{eq:bbw}) satisfies this condition and hence all lemniscate knots are spiral, but in general spiral knots are not subject to (\ref{eq:epscondition}).
Following \cite{betvwy:2010spiral}, the spiral knot with $s$ strands, $r$ repeats with signs determined by $\boldsymbol{\varepsilon}$ is denoted $S(s,r,\boldsymbol{\varepsilon})$, and the lemniscate knot with $s$ strands, $r$ repeats with $\ell$ lobes by $L(s,r,\ell)$; therefore $L(s,r,\ell) \simeq S(s,r,\boldsymbol{\varepsilon})$ for $\boldsymbol{\varepsilon}$ satisfying (\ref{eq:epscondition}).
Spiral knots have several remarkable properties, which do not depend on $\boldsymbol{\varepsilon}$, summarised in the following.

\begin{theorem} 
\label{spiral}The spiral knot/link $S = S(s,r,\boldsymbol{\varepsilon})$ satisfies the following properties. 
\begin{enumerate}[label=(\roman*)]
\item If $r = 1$, $S$ has one component iff it is the unknot; \label{num:th1} 
\item $S$ is an $m$-component link iff $\operatorname{gcd}(s,r) = m$, and in particular $S$ is a knot when $s$ and $r$ are relatively coprime; \label{num:th2}
\item If $r = 2$, then $S$ is a 2-bridge knot (i.e.~rational); \label{num:2-bridge}
\item Every spiral knot is a periodic knot \cite{livingston:1993knot} with period $r$; \label{num:periodic}
\item Every spiral knot is fibred \cite{stallings:1978constructions}; \label{num:fibred} 
\item If $r = 1$, then the word of $S$ can be rearranged (i.e.~all anagrams of the same isogram close to the same knot or link) and for arbitrary $r$, $S$ is the $r$th power of any such rearrangement; \label{num:th3}
\item If $r$ is a prime power, $r = p^k$, and $S$ a 1-component link, then the Alexander polynomial $\Delta_S(t) \equiv (1 + t + \ldots +t^{s-1})^{r-1} \mod p$; \label{num:th4} 
\item If $S$ is a knot, the genus $g$ of $S$ satisfies $\tfrac{1}{2}\operatorname{deg}\Delta_S \le g \le \tfrac{1}{2}(s-1)(r-1)$; \label{num:th5}
\item If $r$ is a prime power and $S$ a knot, $\operatorname{deg}\Delta_{S} = (s-1)(r-1)$ and the genus $g$ of $S$ satisfies $g = \tfrac{1}{2}(s-1)(r-1)$; \label{num:th6}
\item If $r$ is a prime power, the minimal crossing number $c$ of $S$ satisfies $(s-1)(r-1) < c \le (s-1)r$. \label{num:cor1}
\end{enumerate}
\end{theorem}

The proofs of all of these but parts \ref{num:2-bridge} and \ref{num:fibred} are given in \cite{betvwy:2010spiral} (or are straightforward generalisations; part \ref{num:th4} is based on Murasugi's theorem \cite{murasugi:1971periodic}).
Part  \ref{num:th3} is what allows a spiral knot to depend only on $s$, $r$ and $\boldsymbol{\varepsilon}$, and not on the specific ordering of the basic word, which justifies the notation $S(s,r,\boldsymbol{\varepsilon})$.
Part \ref{num:2-bridge} follows from considering a braid diagram as a braid as in (\ref{eq:strand}) as a parametric curve in cylindrical coordinates with $a=b=1$, with angle $h$, radius $2+Y_j(h)$ and height $X_j(h)$; with $X_j$ as the height function, there are $r$ maxima and $r$ minima, so $S$ is 2-bridge if $r = 2$ (equivalent to (\ref{eq:para}) with $X$ and $Y$ exchanged).
More generally, this representation shows the $r$-fold periodicity as a cyclic symmetry generated by a $2\pi/r$ rotation about the axis of cylindrical coordinates. 
Part \ref{num:fibred} follows from Stalling's theorem \cite{stallings:1978constructions} that a knot is fibred if it has a homogeneous braid representation (i.e.~the knot/link can be represented by a word where each generator only appears with the same sign within the word); this follows directly since the braid words of spiral knots are powers of isograms.

Strictly speaking there are two lemniscate knots $L(s,r,\ell)$ for every choice of $s$, $r$ and $\ell$, one for a positive value of $b$ and one for a negative $b$. 
The two knots are mirror images, so that some invariants like the Alexander polynomial, the crossing number, the braid index and the genus do not distinguish them. 
In particular, the statements of Theorem \ref{spiral} are valid for both cases.

Lemniscate knots have the additional symmetry that the $\varepsilon$ of (\ref{eq:epscondition}) is a palindromic vector if $\ell$ is odd, and anti-palindromic if $\ell$ is even, that is $\varepsilon_{j} = (-1)^{\ell+1}\varepsilon_{s-j}$. 
As we show below, this seems to give rise to symmetric tangle representations of rational lemniscate knots, and similarly palindromic minimal braid words (where known), although we do not have a general proof which covers all values of $s$, $\ell$ and $r$.
In common with other studies of spiral knots \cite{kst:2016sequences}, families of lemniscate knots and links seem to have common properties regarding their Alexander polynomial coefficients (for knots), Jones polynomial coefficients and tangle notation (when $r=2$), implying they are worthy of study in general, not simply as the knots simply realizable as nodes of complex scalar functions.  
We arrange our observations by $\ell$, principally considering $r = 2$ (i.e.~the rational knots) and increasing $s$; knots are recognised from standard tabulations \cite{knotatlas,knotinfo} using polynomial invariants. 
The limitations of these tables (going no higher than minimal crossing number $c = 14$ for knots, and $c = 10$ for links) mean that very few lemniscate knots with $r\ge 3$ and $s > \ell+1$ can be identified (although invariants can be calculated for others).

When $\ell = 1$, we have the torus knots, with $L(s,r,1)$ being the $(s,r)$ torus knot (which is isotopic to the $(r,s)$ torus knot).
Since all crossings in the braid words for these have the same sign, the braid words generating the knots are not only homogeneous, but strictly positive.
The properties of torus knots are well-known \cite{kawauchi:1996survey}, and we do not consider them further here.

The next case are the `figure-8 family' of lemniscate knots with $\ell = 2$. 
Since $s$ and $\ell$ must be coprime, the braids have an odd number of strands (starting at $s = 3$), and must be knots as $s$ and $r = 2$ are coprime. 
We have the following Theorem:

\begin{theorem}
\label{rob} The period 2, figure-8 lemniscate knot $L = L(s=2n+1,r=2,\ell=2)$ has minimal braid word $\sigma_1^{-1}\sigma_2^{n}\sigma_1^{-n}\sigma_2$, and has Alexander polynomial given by 
\[\Delta_L(t) = t^{-n} - 3 t^{-n+1} + 5 t^{-n+2} \ldots + (-1)^n s + \ldots + t^n = \sum_{k = -n}^n (-1)^{n+k} (2(n-|k|)+1)t^k.\]
\end{theorem}

The proof is given in the Appendix.
Combining the results of Theorem \ref{spiral} and \cite{king:2010thesis}, we see that these knots are rational, with minimal crossing number $c = s+1$ (whereas the original generating braid has $2s-2$ crossings), braid index $b_{\mathrm{ind}}=2$ (i.e.~number of strands of minimal braid word), $\operatorname{deg}\Delta_{L} = n$ and genus $n$.
Properties of the first few members of this family are given in Table \ref{tab:ell=2}.

\begin{table}[h!]
   \centering
   \begin{tabular}{|l|llllll|}
   \hline
   $s$ & 3 & 5 & 7 & 9 & 11 & 13\\
   \hline
   knot type & 4$_1$ & 6$_3$ & 8$_9$ & 10$_{17}$ & K12a$_{1273}$ & K14a$_{19298}$ \\
   tangle notation & $[22]$ & $[2112]$ & $[3113]$ & $[4114]$ & $[5115]$ & $[6116]$ \\ 
   Jones coeffs & 1,1,1 & 1,2,2,3 & 1,2,3,4,5 & 1,2,3,5,6,7 & 1,2,3,5,7,8,9 & 1,2,3,5,7,9,10,11 \\ 
\hline
\end{tabular}
   \caption{$L(s = 2n+1, r=2,\ell=2)$ lemniscate knots for $n = 1,2,\ldots,6$, identified by knot type, tangle notation and Jones polynomial coefficients.
   For these knots, the Jones polynomial is found to have the same symmetry properties as the Alexander polynomial (i.e.~signs of coefficients alternating, $V(t^{-1}) = V(t)$) and the degree is $n+1$; only the absolute values of coefficients are given, for powers $t^{-n-1}, \ldots, 1$.}
   \label{tab:ell=2}
\end{table}

The tangle notation for $s \ge 5$, always has the symmetric form $[n\ 1\ 1\ n]$.
The symmetries of the braid word for even $\ell$ imply that these knots are achiral, and hence their Jones polynomials $V(t)$ have a similar form to the Alexander polynomial (i.e.~$V(t) = V(t^{-1})$ with alternating signs of coefficients), with the coefficient of the constant term always positive. 
Since 2-bridge knots are alternating, the span of the Jones polynomial is equal to the crossing number.
Furthermore, as $s$ increases, the sequence of coefficients appears to settle to the odd integers including 2, i.e. $1,2,3,5,7,\ldots$ (although we have no general proof).

We previously listed the figure-8 family of knots and links with $s = 3$: again, these are $4_1$ ($r = 2$), $L6a_4$ ($r = 3$), $8_{18}$ ($r = 4$), $10_{123}$ ($r=5$), $\ldots$.
This suggests that for these knots $c = 2r$; some properties of this sequence, such as the values of the determinants $\Delta(-1)$, have been described in \cite{betvwy:2010spiral,kst:2016sequences}.
The only other $\ell = 2$ lemniscate knot (i.e.~$s>3$ and $r>2$) appearing in tabulations is $L(5,3,2)$, which is isotopic to $K12n_{706}$.

For lemniscate knots and links with $\ell = 3$, the $r = 2$ cases for the lowest numbers of strands (coprime to 3) are given in Table \ref{tab:ell=3}.
Both knots and links appear in the list, so Alexander polynomials are not considered. 
For odd $\ell$ the lemniscate knots $L(s,r,\ell)$ are in general not achiral, so the closures of the braids with $b \gtrless 0$ are not mutually isotopic. 
In Table $\ref{tab:ell=3}$ we consider positive values of $b$, as it seems to give rise to a tangle notation where all entries are positive. 
Constructing a similar table for negative values of $b$ from $\ref{tab:ell=3}$ is trivial.
As in the $r = 2, \ell = 2$ case considered above, the tangle representations are all symmetric; for $s \ge 7$ they follow the same pattern $[n11(n-2+m)11n]$ where $s = 3n+m$ (and $m = 1,2$). 
The pattern of crossing numbers in the table suggests that $c=s+2$.
The absolute values of the coefficients of the Jones polynomials form a triangular arrangement of integers, with a maximum at the constant coefficient; as $s$ increases, pattern seems to settle down to sequence $1,4,9,17,29 \ldots$ (as negative powers decrease to $0$), which for $n > 2$ is given by the formula $n^2 + (n - 3)^2$. 
For decreasing positive powers, the coefficients form the sequence $1,3,6,12,21 \ldots$ which for $n > 4$ is given by $n^2 + (n - 3) (n - 7)$.
Minimum braid words have been found for these cases up to $s = 7$, being $\sigma_1\sigma_2^{-1}\sigma_1\sigma_3\sigma_2^{-1}\sigma_3$ for $s = 4$, $\sigma_1\sigma_2^{-1}\sigma_1\sigma_2^{-1}\sigma_3\sigma_2^{-1}\sigma_3$ for $s=5$ and $\sigma_1^{2}\sigma_2^{-1}\sigma_1\sigma_2^{-1}\sigma_3\sigma_2^{-1}\sigma_3^{2}$ for $s=7$.
The braid index in all these cases is equal to $\ell+1=4$.
Increasing the period $r$ with $s = 4$ gives the sequence of knots and links beginning $L6a_1$ ($r=2$), $9_{40}$ ($r=3$), then an untabulated 12-crossing link.
These appear to have $c = 3r$, and the sequence of determinants of any minor of the crossing matrix is found in \cite{kst:2016sequences}.

\begin{table}[h!] 
   \centering
   \begin{tabular}{|l|llllll|}
   \hline
   $s$ & 4 & 5 & 7 & 8 & 10 & 11 \\ 
   \hline
   knot/link & $L6a_1$ & $7_7$ & $9_{31}$ & $L10a_{91}$ & $L12_{??}$ & $13a_{4296}$ \\
   Jones  & 1,2,2,2, & 1,2,3,4, & 1,4,6,8,10, & 1,4,7,10,13, & 1,4,9,15,22,28,30, & 1,4,9,17,26,36,43,45, \\
   $\;$ coeffs & $\;$ 3,1,1 & $\;$ 4,3,3,1 & $\;$ 9,8,5,3,1 & $\;$ 13,9,3,1 & $\;$ 29,25,18,12,6,3,1 & $\;$ 44,37,29,20,12,6,3,1 \\ 
\hline
\end{tabular}
   \caption{$L(s, r=2,\ell=3)$ lemniscate knots for $s$ coprime to $3$, identified by knot type, genus, tangle notation (they are all rational) and absolute values of Jones polynomial coefficients. 
   We cannot identify the link at $s = 10$ as available tabulations do not extend far enough.
   }
   \label{tab:ell=3}
\end{table}

Increasing $\ell$ further gives families which have similar features discussed in \cite{king:2010thesis}. 
These include, for $r = 2$, similar patterns in tangle notation and Jones polynomial coefficients as discussed here for $\ell = 2$ and $3$.

The patterns in crossing numbers, braid index and Conway tangle notation indicated by Tables 2 and 3 generalise to the following result, which holds in general for spiral knots and links.

\begin{theorem}
\label{conway} Let $L$ be a spiral link with $r=2$. 
Then it is rational and if we write the vector \[\boldsymbol{\varepsilon}=(\varepsilon_{1,1},\varepsilon_{1,2},\ldots,\varepsilon_{1,n_{1}},\varepsilon_{2,1},\ldots,\varepsilon_{2,n_{2}},\ldots,\varepsilon_{\ell,n_{\ell}})\] with $\varepsilon_{i,j}=\varepsilon_{i,k}$ for all $j$, $k\in\{1,\ldots,n_{i}\}$ and $\varepsilon_{i,n_{i}}=-\varepsilon_{i+1,1}$, then the Conway tangle notation of $L$ is 
\[
[\varepsilon_{1,1} n_{1},\ \varepsilon_{1,1},\ \varepsilon_{1,1},\ \varepsilon_{1,1}(n_{2}-1),\ \varepsilon_{1,1},\ \varepsilon_{1,1},\ \varepsilon_{1,1}(n_{3}-1),\ldots\ \varepsilon_{1,1}(n_{\ell-1}-1),\ \varepsilon_{1,1},\ \varepsilon_{1,1},\ \varepsilon_{1,1} n_{\ell}].
\]
We also have that the minimal crossing number is $c(L)=s+\ell-1$ and if $L$ is a knot, then the braid index $b_{\mathrm{ind}}$ is $b_{\mathrm{ind}}(L)=\ell+1$.
\end{theorem}

The proof is given in the Appendix.
Recall that a lemniscate knot is just a special case of a spiral knot and that the number of loops $\ell$ is equal to the number $\ell$ defined implicitly in the theorem. 
Since in Table \ref{tab:ell=3} we consider the $b>0$ and $\ell=3$ cases, Theorem \ref{conway} confirms the patterns indicated in the tables. 
While the proof of Theorem \ref{conway} gives a formula for the braid index $b_{\mathrm{ind}}(L)$, it does not provide a form of a braid on $b_{\mathrm{ind}}(L)$ strands. 
From the known braid words for $\ell=3$ we expect the minimal braid word of $L=L(s,r=2,\ell=3)$ to be of the form $ \sigma_{1}^n \sigma_2^{-1} \sigma_{1} \sigma_{2}^{-n+1} \sigma_{3} \sigma_{2}^{-1} \sigma_{3}^n $ if $s=3n+1$ and $\sigma_{1}^n \sigma_{2}^{-1} \sigma_{1} \sigma_{2}^{-n} \sigma_{3} \sigma_{2}^{-1} \sigma_{3}^n $ if $s=3n+2$.

We speculate that any spiral knot with small $r$ (hence for any such lemniscate knot), $c = (r-1)(s + \ell - 1)$, which would generalise known $c$ for torus knots (with $\ell = 1$) and $c$ in Theorem \ref{conway}. 
If the knot is alternating and $r$ is a prime power, using \cite{murasugi:1991braid} and \cite{betvwy:2010spiral}, the braid index then should be of the form $b_{\mathrm{ind}}=(r-1)\ell+1s$. 
While examples for low $r$ are consistent with these formulas, a theorem by Lee and Seo \cite{ls:2010formula} implies that if $r>\max\{n_{1},n_{2}\ldots,n_{\ell}\}$, then $b(L(s,r,\boldsymbol{\varepsilon}))=s$, where $n_{i}$ is defined as in Theorem \ref{conway}.
Hence if $r > \max\{n_{1},n_{2}\ldots,n_{\ell}\}$ is a prime power and $L(s,r,\boldsymbol{\varepsilon})$ is alternating, the crossing number is of the form $c=r(s-1)$. 
This means that in this case the braid diagram in its spiral form minimises both the braid index and the crossing number.
A theorem by Lee \cite{lee:2004alexander} gives an upper bound for $r\leq 2g+1$ where $g$ is genus.
Combining this with Theorem \ref{spiral}\ref{num:th6} and Theorem \ref{conway} it can be shown that $6_2$ (the closure of $(\sigma_{1}\sigma_{2}\sigma_{3}\sigma_{4}^{-1})^{2})$ is spiral, but not lemniscate and $8_{20}$ (the closure of $(\sigma_{1}^{3} \sigma_{2}^{-1})^{2})$ is the closure of a homogeneous braid, but not spiral.

This concludes our discussion of the properties of the lemniscate knots.
We now show how the closures of the corresponding parameteric braids leads to complex maps of which the lemniscate knots are nodal lines.

\section{Complex maps and fibrations for lemniscate knots}\label{sec:fibration}

The construction of polynomial maps $f:\mathbb{R}^3 \longrightarrow\mathbb{C}$ where the nodal lines form lemniscate knots was outlined at the end of Section \ref{sec:intro}.
In this section, we prove that these maps indeed have the desired knotted zeros and are fibrations (for appropriate values of $a$ and $b$ in (\ref{eq:XYsrl})), and then briefly explore some generalisations of the lemniscate knot construction.
With $\ell$, $s$ and $r$ be positive integers with $s$ and $\ell$ coprime, (\ref{eq:XYsrl}) gives a parametrisation $(X_{j}^{s,r}(h),Y_{j}^{s,r,\ell}(h),h)$, where $j=1,2,\ldots,s$ and $h\in[0,2\pi]$ of the lemniscate braid with $\ell$ loops, $s$ strands and $r$ repeats of the basic braid word. 
These give rise to the family of complex polynomials $p^{s,r,\ell}_h$ as in (\ref{eq:polydef}) with roots $Z^{s,r,\ell}_{j}(h) = X^{s,r}_{j}(h) + \mathrm{i}Y^{s,r,\ell}_{j}(h)$. 
By construction, the lemniscate braid in its parametrisation $(X_{j}^{s,r}(h),Y_{j}^{s,r,\ell}(h),h)$ is the preimage of zero of $p^{s,r,\ell}$ as a map from $\mathbb{C}\times[0,2\pi]\longrightarrow\mathbb{C}$.

It can be shown, for example by using induction on the number of factors and elementary arithmetic of roots of unity, that due to the trigonometric form of the roots $Z^{s,r,\ell}_{j}$ the map $p^{s,r,\ell}:\mathbb{C}\times[0,2\pi]\longrightarrow\mathbb{C}$, $p^{s,r,\ell}(u,h)=p^{s,r,\ell}_{h}(u)$ can be written as a polynomial in the variables $u$, $\exp(\mathrm{i}h)$ and $\exp(-\mathrm{i}h)$ (where, of course, $\exp(\mathrm{i}h)\exp(-\mathrm{i}h) = 1$). 
Hence $p^{s,r,\ell}$ can be seen as the restriction of a complex map $f:\mathbb{C}^2\longrightarrow\mathbb{C}$ to the set $\mathbb{C}\times S^{1}$ (where necessary, $S^1$ is now the complex unit circle $\exp(\mathrm{i}h)$) and $f$ is a polynomial in complex variables $u$, $v$ and $\overline{v}$. 
$f$ is derived from $p^{s,r,\ell}$ by writing $v$ for every instance of $\exp(\mathrm{i}h)$ and $\overline{v}$ for every instance of $\exp(-\mathrm{i}h)$ in the polynomial expression of $p^{s,r,\ell}$ (from the construction defined above, no term in the polynomial has any occurrence of $v \overline{v}$). 
In doing so we identify $h=0$ and $h=2\pi$, which closes the ends of the braid. 
Thus we have $f^{-1}(0)\cap(\mathbb{C}\times S^{1}) = K$, where $K$ is the lemniscate knot $L(s,r,\ell)$.

We have not so far considered the positive stretching parameters $a,b$ in (\ref{eq:XYsrl}).
Although the construction is valid for any $a,b>0$, in order to guarantee that we get the desired knot, i.e.~$f^{-1}(0)\cap S^{3}=K$, we need to consider specific choices for these parameters. 
We consider the one-parameter family of parametrisations $\bigcup_{j=1}^{s}\left(X_{j}^{s,r}(h),Y_{j}^{s,r,\ell}(h),h\right)$ where $a = \lambda a_1, b = \lambda b_1$, $\lambda$ is a positive parameter, and $a_1$ and $b_1$ are fixed positive real numbers. 
To indicate the dependence of $Z^{s,r,\ell}_{j}(h)=X_{j}^{s,r}(h)+iY_{j}^{s,r,\ell}(h)$, $p^{s,r,\ell}$ and $f$ on $\lambda$, we will write these maps respectively as $Z_{\lambda,j}$, $p_{\lambda}$ and $f_{\lambda}$ (suppressing other suffixes).

We want to show that, for small enough $\lambda$, $f^{-1}(0)\cap S^{3}$ is isotopic to the desired knot $K$, using the complex coordinates of the unit 3-sphere, $S^3 = \{(u,v); |u|^2 + |v|^2 = 1\}$. 
Note in particular that for sufficiently small $\lambda$, $|Z_{\lambda,j}(h)| < 1$, i.e.~for all $j$, $Z_{\lambda,j}(h) \in \mathbb{D} \equiv \{u \in \mathbb{C}; |u|<1\}$.
Thus the image of $f_{\lambda}^{-1}(0)\cap(\mathbb{C}\times S^{1})$ under the map $P:\mathbb{D}\times(\mathbb{C}\backslash \{0\})\longrightarrow S^{3}$ (where $S^1 = \{\mathrm{e}^{\mathrm{i} h}; h\in [0,2\pi]\}$),
\begin{equation}
	P(u,\rho\mathrm{e}^{\mathrm{i}h})=(u,\sqrt{1-|u|^{2}}\mathrm{e}^{\mathrm{i}h})
	\label{eq:proj}
\end{equation}
is ambient isotopic to $K$ (the image of $P$ is independent of the modulus $\rho > 0$ of the second argument).
This result can be easily seen when $K$ is given in the parametrisation (\ref{eq:para}).
We now outline a proof of the ambient isotopy from the image under $P$ of $f_{\lambda}^{-1}(0)\cap(\mathbb{C}\times S^{1})$, known to be $K$, to $f_{\lambda}^{-1}(0)\cap S^{3}$ for small enough $\lambda$.

\begin{theorem}
\label{isotopy} For all positive integers $\ell$, $s$ and $r$ with $s$ and $\ell$ coprime and all choices of $a>0$ and $b>0$, there exists an $\epsilon$ such that $f_{\lambda}^{-1}(0)\cap S^{3}=K$ for all $\lambda<\epsilon$, where $K$ is the $(s,r,\ell)$ lemniscate knot.
\end{theorem} 
\begin{proof}
First note that by definition $Z_{\lambda,j}(h)=\lambda Z_{1,j}(h)$. 
Furthermore, for all fixed $\lambda$ and $v$ the function $f_{\lambda}(\cdot,v)$ is a polynomial in $u$ of degree $s$. 
We extend the definition of $Z_{\lambda,j}(\rho,h)$ with $j=1,2,\ldots,s$ to be the roots of $f_{\lambda}(\bullet,\rho \mathrm{e}^{\mathrm{i}h})$. 
This allows us to choose $\lambda$ small enough such that $Z_{\lambda,j}(\rho,h)\in\mathbb{D}$.

We consider the different $Z_{\lambda,j}(\rho,h)$ as functions of $\rho$.
These are are smooth at $\rho=1$, since for every $h$, the roots $Z_{\lambda=1,j}(1,h)$ of $f_{\lambda=1}(\bullet,\mathrm{e}^{\mathrm{i}h})$ are disjoint and disjoint roots of polynomials depend smoothly on coefficients. 
Hence there is a $\delta>0$ independent of $\lambda$ such that $Z_{\lambda,j}(\rho_{1},h)\neq Z_{\lambda,k}(\rho_{2},h)$ for all $j\neq k$ and $\rho_{1},\ \rho_{2}\in[1-\delta,1]$.
Now, we can choose $\epsilon$ such that for any $\lambda<\epsilon$, for every $h$ and every $j$ the curve $(Z_{\lambda,j}(\rho,h), \rho \mathrm{e}^{\mathrm{i}h})$ intersects $S^{3}$ at a unique point $(Z_{\lambda,j}(\rho_{\lambda,j,h},h), \rho_{\lambda,j,h}\mathrm{e}^{\mathrm{i}h})$ and $\rho_{\lambda,j,h} \in [1-\delta,1]$.
Then define $\Psi:P(f_{\lambda}^{-1}(0)\cap(\mathbb{C}\times S^{1}))\times[0,1] \longrightarrow S^{3}$ to be
\begin{equation}
	\Psi(P(Z_{\lambda,j}(1,h),\mathrm{e}^{\mathrm{i}h}),t) = P(Z_{\lambda,j}(t \rho_{\lambda,j,h} - (t-1)\rho,h),\rho \mathrm{e}^{\mathrm{i}h})
	\label{eq:isotopy}
\end{equation} 
By construction $\Psi$ is a smooth isotopy from $P(f_{\lambda}^{-1}(0)\cap(\mathbb{C}\times S^{1}))=K$ to $f_{\lambda}^{-1}(0)\cap S^{3}$ if $\lambda<\epsilon$. 
By the isotopy extension theorem, it extends to an ambient isotopy \cite{ek:1971imbeddings} showing that $f_{\lambda}^{-1}(0)\cap S^{3}=K$.
\end{proof}

Theorem \ref{isotopy} shows that the zero set of $f_{\lambda}$ restricted to $S^3$ has the desired knot type provided $\lambda$ is small enough. 
The proof does not specify how small $\lambda$ has to be, since it does not give a value for $\epsilon$. However, values can be calculated using bounds on the modulus of roots of polynomials and the implicit function theorem.
Note in particular that $\epsilon$ does not depend on $r$.
We have checked numerically for the explicit examples in Section \ref{sec:braids}, with $a_1=b_1=1$; for the $\ell = 2$ (Table \ref{tab:ell=2}), $\lambda \le 1$ is sufficient, and for the $\ell=3$ (Table \ref{tab:ell=3}), $\lambda \le 1/2$ is sufficient (from the numerical behaviour, we suspect these will suffice for higher $s$).

Since $f_{\lambda}$ is a polynomial in $u$ and $Z_{\lambda,j}(\rho,h)\neq Z_{\lambda,k}(\rho,h)$ for all $\lambda$, $j\neq k$, $h$ and $\rho\in[1-\delta,1]$, the points on the intersection of $f_{\lambda}^{-1}(0)$ and $S^{3}$ are regular points of $f_{\lambda}$, so $\nabla_{\mathbb{R}^{4}} f_{\lambda}$ has full rank (here and below, $\nabla_M$ denotes the gradient map on a manifold $M$).
Since the intersection of $f_{\lambda}^{-1}(0)$ with $S^{3}$ is transverse, the knot is in fact a set of regular points of the restriction of $f_{\lambda}$ to the unit 3-sphere, i.e.~$\nabla_{S^{3}} f_{\lambda}\neq 0$. 
This allows small smooth perturbations of the coefficients without altering the link type of the nodal set. 
This is particularly advantageous when additional physical constraints have to be taken into account.

The transversality of the intersection also builds a connection to the notion of transverse $\mathbb{C}$-links. 
These were defined by Rudolph \cite{rudolph:2005knot} to be the links that arise as transverse intersections of a complex plane curve and the unit 3-sphere. 
In our case, we do not deal with complex plane curves, but with zero sets of semiholomorphic polynomials, a significantly weaker notion.

Recall from Theorem \ref{spiral}\ref{num:fibred} that lemniscate knots are fibred. 
Having constructed a polynomial $f_{\lambda}:S^{3}\longrightarrow\mathbb{C}$ with $f_{\lambda}^{-1}(0)=K$, one might ask whether the map $\arg(f_{\lambda}):S^{3}\backslash K\longrightarrow S^{1}$ is a fibration. 
By the Ehresmann Fibration Theorem \cite{ehresmann:1950connexions} it is sufficient to check that the phase function $\arg(f_{\lambda})$ does not have any critical points, i.e.~$\nabla_{S^{3}} \arg(f_{\lambda})\neq 0$.
We have the following result.

\begin{theorem}\label{critical}
If $a,b>0$ are such that $p_{\lambda}(u,h)=\prod_{j=1}^{s}(u-Z_{\lambda,j}(h))$ does not have any phase-critical points $x\in\mathbb{C}\times[0,2\pi]$, i.e.~no points $x_{0}$ at which $\nabla_{\mathbb{C}\times[0,2\pi]} \arg(p_{\lambda})|_{x=x_{0}} = 0$ (for one value of $\lambda>0$ and equivalently for all $\lambda>0$), then there is no point $x_{0}'\in S^{3}\backslash K$ with such that at $x_{0}'$, $\nabla_{S^{3}}\arg(f_{\lambda})|_{x=x_{0}'}=0$ for all small enough $\lambda>0$. 
Hence $f_{\lambda}$ induces a fibration over $S^{1}$.
\end{theorem}

\begin{proof}

Note that the derivative $\partial \arg(f_{1})/\partial \arg(u)$ converges uniformly to $s$ on $(u,v)\in\{\mathbb{C}^2 : |u|=R\}$ as $R\to\infty$. 
In particular, at some $R>0$, $\partial \arg(f_{1})/\partial \arg(u)>0$ when evaluated at all $(u,\ v)\in\mathbb{C}^2$ with $|v|\leq 1$ and $|u|>R$.
This means that for $f_{\lambda}$ with $\lambda>0$ the same statement holds for all $|u|>\lambda R$. 

Since $p_{\lambda}$ does not have any phase-critical points and $\partial \arg\left(f_{\lambda}(u,\mathrm{e}^{\mathrm{i}h})\right)/\partial \arg(u) =\partial \arg\left(p_{\lambda}(u,h)\right)/\partial \arg(u)$, it follows from the continuity of $\nabla_{\mathbb{R}^4} \arg(f_{\lambda})$ away from the zeros of $f_{\lambda}$ that there is a $\epsilon>0$ such that $\left|\nabla_{S^3} \arg(f_{\lambda})(u,v)\right|>0$ for all $v$ with $1-\epsilon<|v|\leq 1$. 
Now choose $\lambda\in\mathbb{R}$ such that $\lambda^2  R^2 + (1-\epsilon)^2<1$ and it follows that $\left|\nabla_{S^3\backslash K} \arg(f_{\lambda})\right|>0$ everywhere, as either $|u|>\lambda R$ or $1-\epsilon<|v|\leq 1$. 
Hence for all sufficiently small $\lambda$, $f_{\lambda}$ does not have any phase-critical points on $S^{3}\backslash K$.
\end{proof}

Theorem \ref{critical} gives a sufficient condition for an explicit fibration of the knot complement over $S^1$ as the argument of a semiholomorphic polynomial. 
We are not aware of any procedure that would find values for $a$ and $b$, such that this condition is satisfied for $p_{\lambda}$. 
In fact, it is not even clear if such values always exist. 
We have checked numerically for the explicit lemniscate knots identified in Section \ref{sec:braids}, $a=b=\lambda\le 1$ is sufficient, meaning the maps constructed for them (with the previous values of $\lambda$) are, indeed, fibrations.

We have proved, therefore, that the nodal lines of the functions $f_{\lambda}:\mathbb{R}^3 \longrightarrow \mathbb{C}$, for small enough $\lambda$, indeed form the $(s,r,\ell)$ lemniscate knot or link, and for the explicit knots considered in Section \ref{sec:braids}, these are in fact explicit fibrations.
In spite of these technical details, the procedure for constructing complex functions with zeros in the form of lemniscate knots (or indeed fibrations of lemniscate knots) is just that in Section \ref{sec:intro} with appropriate choices of $a$ and $b$.
For low $s$, we found that it is sufficient that $a,b \le 1$ if $\ell = 2$, and $a,b\le 1/2$ if $\ell = 3$.
The braid polynomial $p$ is constructed as from Equation \ref{eq:polydef} and then multiplied out and simplified. Then all occurrences of $\exp(\mathrm{i} h)$ are replaced by $v$, and $\exp(-\mathrm{i} h)$ by $v$.
Rewriting $u$ and $v$ in terms of $x,y,z$ according to (\ref{eq:coords}) gives an explicit complex polynomial of three-dimensional real space with a nodal knot or link.
If $a$ and $b$ are rational, all coefficients are rational, so by multiplying by a constant, we can make sure that $f$ has integer coefficients. 
This generalises the procedure for the figure-8 knot $L(3,2,2)$ whose function $f_{\mathrm{fig}\mbox{-}8}$ was given in (\ref{eq:ffig8}).
Other examples are
\begin{align}
f^{(5,r,2)}(u,v)& = 1024 u^5 - 960 u^3 - 160 u^2 (v^r + \overline{v}^r) + 20 u (21 - 10 (v^r + \overline{v}^r)) - 82 (v^r + \overline{v}^r) - v^{2r} + \overline{v}^{2r}, \label{eq:f5r2} \\
f^{(4,r,3)}(u,v)& = 20736 u^4 - 576 u^2 (8 + 3 (v^r - \overline{v}^r)) + (92 - 39 v^r - 231 \overline{v}^r + 6 v^{2r} + 30 \overline{v}^{2r} - v^{3r} - \overline{v}^{3r}). \label{eq:f4r3}
\end{align}
This method of constructing polynomials with knotted nodal lines works for every lemniscate knot.
It can be easily seen from (\ref{eq:XYsrl}) and the definitions of $p$ and $f$ that the degree of $f$ in $u$ is equal to $s$ (the number of strands in the braid), and its degree in $v$ and $\overline{v}$ is equal to $r\ell$.

All results so far apply to lemniscate knots, which form a very restricted family of knots with particular symmetries as discussed in Section \ref{sec:braids}. 
However, many statements here extend to knots which are not closures of braids with a lemniscate parametrisation. 
We will briefly consider two further families here.

Instead of starting with a braid where each strand follows a fixed lemniscate figure, we can consider braids where the strands follow a \emph{rotating} lemniscate figure. 
The $s$ strands are then parametrised by 
\begin{equation}
	Z_{j}^{s,r,\ell,n}(h)=e^{\mathrm{i}hn}(X_{j}^{s,r}(h)+\mathrm{i}Y_{j}^{s,r,\ell}(h)), 
	\label{eq:rotate}
\end{equation}
where $n\in\mathbb{Z}$ is the number of full clockwise rotations that the lemniscate figure performs as $h$ increases from $0$ to $2\pi$. 
Again the function $p_{\lambda}(u,h)=\prod_{j=1}^{s}(u-\lambda Z_{j}^{s,r,\ell,n}(h))$ is a polynomial in $u$, $\mathrm{e}^{\mathrm{i}h}$ and $\mathrm{e}^{-\mathrm{i}h}$ and the proof of Theorem \ref{isotopy} implies that for small enough $\lambda$, replacing $e^{\mathrm{i}h}$ by $v$ and $e^{-\mathrm{i}h}$ by $\overline{v}$ results in a semiholomorphic polynomial $f_{\lambda}$ such that its nodal set on $S^{3}$ is ambient isotopic to the closure of the braid with strands defined by (\ref{eq:rotate}).

\begin{figure}
\begin{center}
\begin{subfigure}[t]{0.4\textwidth}
\centering
\includegraphics[height=4.5cm]{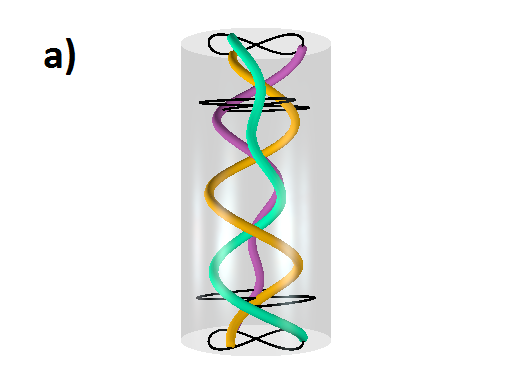}
\end{subfigure}
\begin{subfigure}[t]{0.4\textwidth}
\centering
\includegraphics[height=4.5cm]{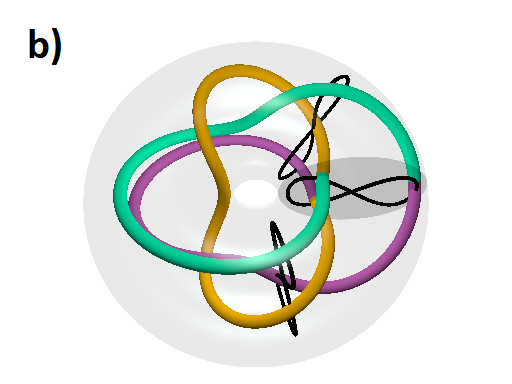}
\end{subfigure}
\caption{
   Rotating lemniscate braid closing to the composite of two trefoil knots. 
   (a) Braid executing a rotating lemniscate, $(s,r,\ell,n) = (3,2,2,1)$, with braid word $(\sigma_1^{-1}\sigma_2)^2 (\sigma_{1}\sigma_{2}\sigma_{1})^{-2}$. 
   (b) Closure of the braid forming the composite of two trefoil knots (with same chirality), with minimum braid word $\sigma_1^{-3}\sigma_2^{-3}$.  
}
\label{fig:rotating}
\end{center}
\end{figure}

In the case of $s=3$, $r=2$, $\ell=2$ and $n=1$, this results in the connected sum of trefoil knots as can be seen in Figure \ref{fig:rotating}. 
It is not hard to show that in general rotating the braid which closes to the $L(s,r,\ell)$ by $n$ turns, is isotopic to the braid with word $w^{s,r,\ell}\Delta_{s}^{-2n}$ where $\Delta_{s}$ is the Garside element of the braid group $B_{s}$ given by $\Delta_{s}=(\sigma_{1})(\sigma_{2}\sigma_{1})(\sigma_{3}\sigma_{2}\sigma_{1})\ldots(\sigma_{s-1}\ldots\sigma_{1})$.

Alternatively, more complicated braids can be built up by a generalisation of the notion of cabling (indeed, generalising the torus knot construction by cabling in the context of polynomials with knotted nodal lines goes back to Brauner \cite{brauner:1928geometrie}). 
Rather than considering braids where single strands execute generalised lemniscate trajectories, a single strand/root in the polynomial may be replaced by a cluster of strands/roots, executing their own figure which we call an `epicycle': such an epicycle of strands may execute a generalised lemniscate figure, within a larger figure. 
This process can be iterated; however, closures of such braids typically have many crossings, so this does not generate tabulated knots.
One example which is tabulated is the simplest satellite knot $13n_{4587}$, formed as a cable of a $(3,2)$ trefoil knot, where the original trefoil strands are replaced by a pair of strands executing a $\ell=1$ circle through a quarter turn as $0 \le h \le 2\pi$.
With the outer radius unity, inner epicycle radius $1/4$, this braid has polynomial 
\begin{equation}
   p_h(u) = 
   \left(u - (\mathrm{e}^{3\mathrm{i}h/2} + \tfrac{\mathrm{i}}{4} \mathrm{e}^{\mathrm{i}h/4})\right)
   \left(u - (\mathrm{e}^{3\mathrm{i}h/2} - \tfrac{\mathrm{i}}{4} \mathrm{e}^{\mathrm{i}h/4})\right)
   \left(u + (\mathrm{e}^{3\mathrm{i}h/2} + \tfrac{\mathrm{1}}{4} \mathrm{e}^{\mathrm{i}h/4})\right)
   \left(u + (\mathrm{e}^{3\mathrm{i}h/2} - \tfrac{\mathrm{1}}{4} \mathrm{e}^{\mathrm{i}h/4})\right).
   \label{eq:cable}
\end{equation}
This leads to the function $f(u,v) = 256 u^4 - 512 u^2 v^3 + 64 u v^2 - v + 256 v^6$ (which also gives a fibration) whose nodal line is the desired cable knot. 
This example is illustrated in Figure \ref{fig:cable}.

\begin{figure}
\begin{center}
\begin{subfigure}[t]{0.4\textwidth}
\centering
\includegraphics[height=4.5cm]{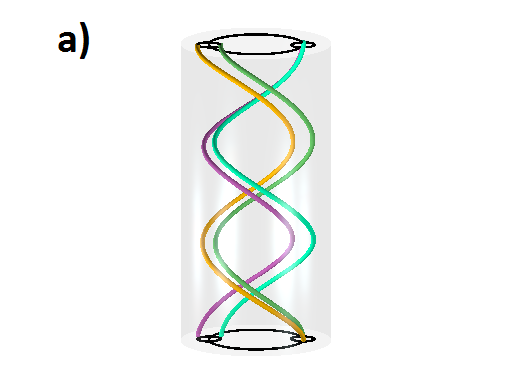}
\end{subfigure}
\begin{subfigure}[t]{0.4\textwidth}
\centering
\includegraphics[height=4.5cm]{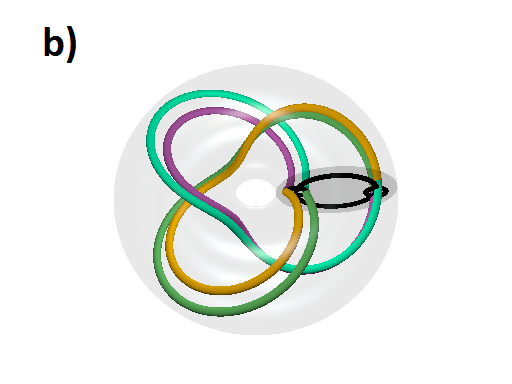}
\end{subfigure}
\caption{Cabling a helical braid and its torus closure.
(a) A $r=3$ double helix braid has its roots replaced by an `epicycle' of two strands, transversally executing a circular trajectory with a quarter turn, given by the roots of (\ref{eq:cable}). 
(b) Closure of the braid, isotopic to the simplest satellite knot $13n_{4587}$.}
\label{fig:cable}
\end{center}
\end{figure}

One can consider many other generalisations of the construction, leading to other families of knots realisable as nodal lines of complex maps, or fibrations.

\section{Physical examples of lemniscate knotted fields}\label{sec:applications}

An outstanding problem in understanding the theoretical implications of knots in physics is the creation of knots as initial conditions in physical systems. 
The complex maps $f(u,v,\overline{v})$ (with appropriate values of $a$ and $b$), and particularly the related $F(\boldsymbol{r})$ (obtained from $f$ using the substitution (\ref{eq:coords}) and taking the numerator) are a set of model knotted fields, with the $(r,s,\ell)$ lemniscate knot or link as nodal line.
$F$ is a polynomial expression in $x$, $y$ and $z$, and therefore quite amenable to analytic or numerical approaches.
We will briefly discuss knotted scalar fields in quantum mechanics and optics, and then present a more extended discussion of knotted hopfions extending the construction to rational maps.

For instance, a knotted nodal lines can be realised in a complex wavefunction $\psi(\boldsymbol{r})$ of quantum mechanics by multiplying $F$ by a positive definite, normalizable weight function $W(\boldsymbol{r})$ such as a gaussian, e.g.~$\psi(\boldsymbol{r}) = F(\boldsymbol{r}) \mathrm{e}^{-|\boldsymbol{r}|^2}$.
Such knots are not expected to persist dynamically, and dissolve under evolution due to the Schr\"odinger equation \cite{king:2010thesis}.
Knots can also be directly embedded into static bound-state wavefunctions in systems whose eigenfunctions are polynomials times a weight function, such as the three-dimensional harmonic oscillator (for which $W(\boldsymbol{r})$ is again a gaussian) or the hydrogen atom (for which $W(\boldsymbol{r}) = \mathrm{e}^{-|\boldsymbol{r}|}/|\boldsymbol{r}|$).
These knots are a sum of nondegenerate energy eigenstates and therefore will not persist in time, unlike knotted nodal lines found in these systems either by construction \cite{berry:2001hydrogen} or by random sampling \cite{td:2016tangled}.

\begin{figure}
\begin{center}
\begin{subfigure}[t]{0.4\textwidth}
\centering
\includegraphics[height=4.5cm]{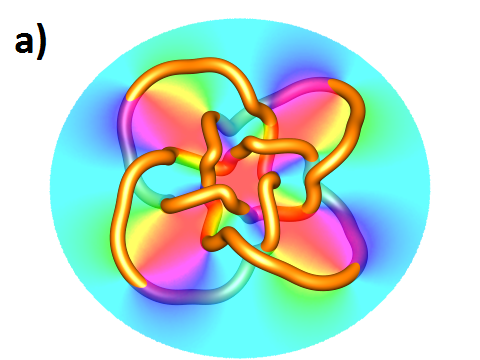}
\end{subfigure}
\begin{subfigure}[t]{0.4\textwidth}
\centering
\includegraphics[height=4.5cm]{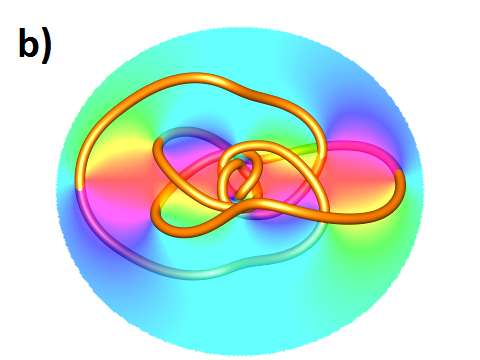}
\end{subfigure}
\caption{
   Knotted optical vortices propagating from holograms given by $F(x,y,0)$.
   In each case, the nodes of the polynomial solution of the paraxial equation coinciding with $F(x,y,z=0)$ are shown, with the coloured disk representing, by hue, the phase (argument) $\arg F(x,y,0)$ of the hologram.
   (a) The knot $8_{18}$ corresponding to lemniscate knot $L(3,3,2)$. 
   (b) The knot $7_7$ corresponding to lemniscate knot $L(5,2,3)$.  
}
\label{fig:prop}
\end{center}
\end{figure}

The motivation of the knot construction of \cite{dkjop:2010isolated} was to embed knotted nodal lines (optical vortices) into the complex amplitude of a paraxial light beam such as that of structured light from a laser.
In this case, the knotted fields $F(\boldsymbol{r})$ does not satisfy the PDE of paraxially propagating light (i.e.~the $2+1$ Schr\"odinger equation with propagation direction $z$ as time); however, it was found for several examples of torus knot and the figure-8 knot that taking a central plane $F(x,y,0)$ of the knotted field, and propagating this polynomial field forwards and backwards \cite{dgkma:2011polynomial} gives a new polynomial $\widetilde{F}(\boldsymbol{r})$ which has a nodal line in the form of the same knot.
We do not have a proof of this behaviour.
In some cases where the correct knot does not occur by evolving $F(x,y,0)$, the correct knot occurs by evolving the `overhomogenised' field $F(x,y,0)(1+x^2+y^2)^n,$ for some choice of positive integer $n$ \cite{king:2010thesis}; again we have no proof of this.
Such propagating polynomial beams can then be embedded in a normalizable beam by multiplying by a weight function $W(x,y)$ (such as a gaussian), and this approach led to the successful experimental generation of several nodal knots \cite{dkjop:2010isolated}.
Two further theoretical examples are shown in Figure \ref{fig:prop}.

Physical applications are not limited to complex scalar fields of the form $\mathbb{R}^3 \longrightarrow \mathbb{C}$.
Many examples, particularly in condensed matter physics or field theory, utilise maps between spaces which have a richer topological structure such as a map to the unit 2-sphere, $\boldsymbol{\varphi}:\mathbb{R}^3 \longrightarrow S^2$.
The field is a three-component unit vector $\boldsymbol{\varphi}=(\phi_1,\phi_2,\phi_3)$, and finite energy requires that the field tends to a constant at spatial infinity, chosen to be $\boldsymbol{\varphi}^\infty=(0,0,1)=\boldsymbol{e}_3$.
This condition compactifies real space to $S^3$, hence the field can be extended to the map $\boldsymbol{\varphi}:S^3 \longrightarrow S^2$. 
Such maps are elements of $\pi_3(S^2)\cong\mathbb{Z}$, and so each $\boldsymbol{\varphi}$ is indexed by an integer $Q$, referred to as the \emph{topological charge}.
Generically, preimages of points on $S^2$ are unions of disjoint loops in $\mathbb{R}^3$ or $S^3$. 
It has been shown \cite{bt:1982} that $Q$ is equal to the linking number of any pair of loops which are preimages of two distinct points on $S^2$.
Such fields $\boldsymbol{\varphi}$ occur in the Skyrme-Faddeev model in particle physics \cite{faddeev:1975quantization}. The fields of the form $\boldsymbol{\varphi}$ which minimise a certain nonlinear energy functional are topological solitons called hopfions.
The positions where the field takes the value furthest from the vacuum $\boldsymbol{\varphi} = (0,0,-1)$ is identified as the hopfion.
Substantial numerical work, e.g.~\cite{sutcliffe:2007knots,hs:2000ground} has shown that the model has hopfion solutions where hopfions can be unknotted loops, links or knots.

A method to generate initial conditions with torus knotted hopfions (non-minimised) was given by Sutcliffe \cite{sutcliffe:2007knots}. 
Real space $\mathbb{R}^3$ is identified with $S^3 \subset \mathbb{C}^2$ using the degree-one spherically-equivariant map,
\begin{equation}
   (u,v)=\left(\cos d(r) +i \frac{\sin d(r)}{r} z, (x + \mathrm{i} y) \frac{\sin d(r)}{r}\right),	
   \label{eq:UVdef}	
\end{equation}
where $\boldsymbol{r} = (x,y,z)\in \mathbb{R}^3$, $r=|\boldsymbol{r}|$ and $d(r)$ is a monotonically decreasing profile function, with the boundary conditions $d(0)=\pi, d(\infty)=0$.
As in previous Sections we describe $S^{3}$ by $\{(u,v)\in\mathbb{C}^{2}:|u|^{2}+|v|^{2}=1\}$. 
The only difference lies in the different choice of projection map, \ref{eq:UVdef} rather than \ref{eq:coords}.
We then define the map $W:\mathbb{C}^2\longrightarrow\mathbb{CP}^1$ as the rational map
\begin{equation}
   W=\frac{g(u,v)}{f(u,v)}, \label{eq:rationalmap}	
\end{equation}
of two complex-valued polynomials $f(u,v), g(u,v)$. 
Stereographic projection gives a Riemann sphere coordinate, $W$, on the target $2$-sphere of $\boldsymbol{\varphi}$, by $W = (\phi_1+\mathrm{i} \phi_2)/(1+\phi_3)$, so the map $\boldsymbol{\varphi}:S^3\longrightarrow S^2$ is equivalent to the rational map $W:S^3\longrightarrow\mathbb{CP}^1\cong S^2$.

With this setup, the knottedness of the configuration follows from an appropriate choice of $f(u,v)$ in (\ref{eq:rationalmap}), since nodes of $f$ correspond, by stereographic projection, to loci where $\boldsymbol{\varphi} = (0,0,-1)$.
The simplest example is the axially-symmetric planar ring for which $W= v^n / u^m$, which has topological charge $Q=nm$. 
The rational map $W= (v^\alpha u^\beta)/(v^q - u^p)$ \cite{sutcliffe:2007knots} has a hopfion in the form of a $(p,q)$ torus knot (assuming $p$ and $q$ are coprime positive integers), since the denominator has the form of the Brauner map (\ref{eq:brauner}).
Assuming $\alpha$ and $\beta$ are integers, these maps have topological charge $Q=\alpha p + \beta q$.

We now extend this rational map construction to lemniscate knotted initial configurations for hopfions.
To achieve this we use the lemniscate polynomials, $f(u,v)=f_{\lambda}(u,v,\overline{v})$, in the denominator of the rational map (\ref{eq:rationalmap}).
We choose the numerator $g$ to be a constant times an integer power of $v$, say $v^N$. 
The boundary conditions are satisfied for any such $N$.
The topological degree of such a map can be calculated from the number of strands. 
As in \cite{sutcliffe:2007knots}, the topological charge $Q$ of $\boldsymbol{\varphi}$ is equal to the topological degree of $(g(u,v),f(u,v))$ extended to a function on the unit 4-ball $\mathbb{B}^4\subset\mathbb{C}^2 \longrightarrow\mathbb{C}^2$. 
By definition, this can be calculated by counting the number of preimages of a regular value, say $(\epsilon,0)$, weighted by the sign of the jacobian. 
Since $f$, the constructed polynomial from Section \ref{sec:fibration}, is holomorphic in $u$ and additionally $g$ does not depend on $u$ and is holomorphic in $v$, the Cauchy-Riemann equations imply that the sign of the jacobian is the same for every point in the preimage of the regular point $(\epsilon,0)$. 
Hence, for $g=v^N$, $Q$ is equal to the number of points that get mapped to $(\epsilon,0)$ by $(g(u,v),f(u,v))$, which is equal to 
\begin{equation}
 Q=\deg_v (g)  \deg_u (f)=Ns, \label{eq:degree}
\end{equation}  
where $s= \deg_u (f)$ is the number of strands used in the construction of $f$. 
This procedure gives knotted field configurations with $Q=Ns$ for any lemniscate knot and any integer $N$.
Furthermore, using $f(u,v)=\left(f_{\lambda}(u,v,\overline{v})\right)^{m}$ for some integer $m>1$ gives initial condition where $m$ copies of the same  knot coincide, i.e.~for every fixed $v$, the corresponding polynomial $f(\bullet,v)$ has $s$ distinct roots of multiplicity $m$.

Generalising our earlier figure-8 knot function (\ref{eq:ffig8}) to $r$ repeats,
we can construct the set of rational maps
\begin{equation}
W=\frac{64 v^N}{64u^3-12 u(2v^r-2\overline{v}^r + 1.5 ) - (14v^r + 14\overline{v}^r+v^{2r}-\overline{v}^{2r})}. \label{eq:fig8hopf}
\end{equation}
Changing the constant $3$ in the second summand of (\ref{eq:ffig8}) to $1.5$ and choosing $64$ in the denominator was found to give an initial condition with a clearer spatial distribution.
Figure \ref{hopf-im} shows the hopfion initial conditions of Equation (\ref{eq:fig8hopf}) for $r = 2$ and $3$.
These configurations do not themselves minimise the Skyrme-Faddeev energy function, and in fact the borromean rings ($r=3$) minimise to a hopfion in in the form of a $(3,4)$ torus knot, shown in Figure \ref{hopf-im} c).

\begin{figure}
       \centering
        \begin{subfigure}[b]{0.3\textwidth}
                \centering \includegraphics[width=\textwidth]{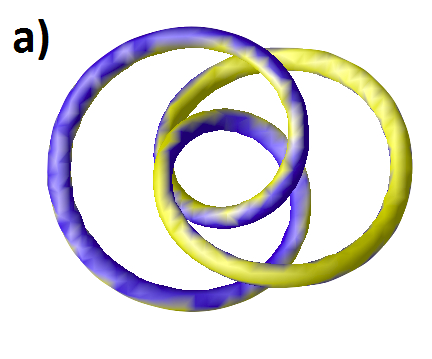}    
        \end{subfigure}
       \begin{subfigure}[b]{0.3\textwidth}
               \centering
    \includegraphics[width=\textwidth]{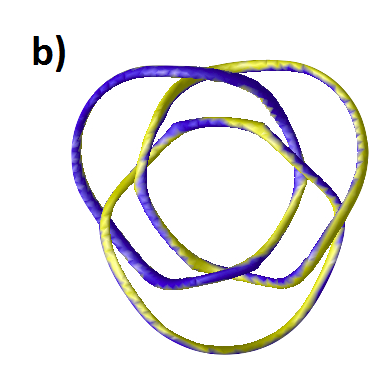}
       \end{subfigure}
        \begin{subfigure}[b]{0.3\textwidth}
                \centering \includegraphics[width=\textwidth]{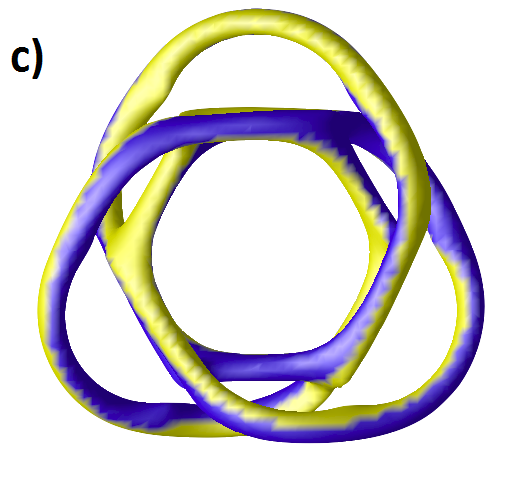}
         
        \end{subfigure}
        \caption{Lemniscate knotted initial conditions for hopfions using (\ref{eq:fig8hopf}) and a corresponding energy minimising torus knot. 
        Blue and yellow are preimages of different points of $S^2$ (blue is the south pole).
        (a) $r=2$, $N=2$; (b) $r = 3$, $N=1$; (c) The energy minimiser, when the borromean rings ($r=3$, $N=1$) are used as initial condition.
        }
        \label{hopf-im}
\end{figure}

The rational map construction $W = v^N/f(u,v,\overline{v})$ can also be used to create real vector fields representing vorticity which contain lemniscate knots, whose helicity is given by the topological charge $Q = Ns$ \cite{kfdi:2016weaving}.

\section{Lemniscate knots as real algebraic knots}\label{sec:AK}

The polynomials studied by Brauner, Milnor, Akbulut and King, Perron and Rudolph all have stronger properties than those constructed in Section \ref{sec:fibration}. 
Their interest in polynomials with knotted zeros was driven by the study of isolated singularities. 
Intersecting the zero set of these polynomials with a 3-sphere around the singularity yields the desired link not only for the unit 3-sphere, but for any 3-sphere of small enough radius.

The point $(x_{1},x_{2},x_{3},x_{4})\in \mathbb{R}^{4}$ is called a \emph{singular point} of a polynomial $f:\mathbb{R}^{4}\longrightarrow\mathbb{R}^{2}$ if the jacobian of $f$ does not have full rank at that point. 
A singular point is said to be \emph{isolated} (respectively weakly isolated) if there is a $\epsilon$-neighbourhood of $(x_{1},x_{2},x_{3},x_{4})$ within which it is the only singular point of $f$ (respectively such that it is the only singular point of $f$ in that neighbourhood that is part of the nodal set).  
Brauner's polynomials $f=u^{p}-v^{q}$ have an isolated singularity at $(0,0)\in\mathbb{C}^{2}$ and the zero set of $f$ on three-spheres of small radius is the same link $L$, in this case the $(p,q)$-torus link.

Links for which there exist such complex polynomials $\mathbb{C}^{2}\longrightarrow\mathbb{C}$ are called \emph{algebraic links} or links of isolated complex singularities. 
Similarly, for real polynomials $\mathbb{R}^{4}\longrightarrow\mathbb{R}^{2}$, the links for which such functions exist are called \emph{real algebraic} or links of isolated real singularities. 
It is known that algebraic links are iterated cables of torus knots satisfying an extra positivity condition (\cite{brauner:1928geometrie,milnor:1968singular,le:1972noeuds,burau:1933knoten,burau:1934verkettunen,zariski:1932topology,kahler:1929verzweigung}). 
A very good survey of the topic of this classification of algebraic links can be found in \cite{en:1985three}.

Milnor \cite{milnor:1968singular} showed that for any complex polynomial $f$ like this, $f/|f|$ defines a fibration of $S_{\epsilon}^{3}\backslash L$ over $S^{1}$. 
For the real case, i.e.~real polynomials $f:\mathbb{R}^{4}\longrightarrow\mathbb{R}^{2}$ with an isolated singularity at $0\in\mathbb{R}^{4}$, the question of which links arise as links of singularities is to our knowledge still open. 
Milnor's result however carries over to isolated real singularities. 
Links of isolated real singularities are fibred links and clearly complex algebraic links must also be real algebraic. 
Explicitly constructing real polynomials with isolated singularities and links of a given type around them has proven to be challenging with Perron's \cite{perron:1982noeud} and Rudolph's \cite{rudolph:1987isolated,rudolph:1984review} constructions of the figure-8 knot, which is not a complex algebraic knot, and a construction of certain odd knots by Looijenga \cite{looijenga:1971note} being the only ones of which we are aware.
Weakening the condition on the singularity of $f:\mathbb{R}^{4}\longrightarrow\mathbb{R}^{2}$, allows it to be weakly isolated, and all links arise as links of such singularities, as shown by Akbulut and King \cite{ak:1981all}.
Thus for every link $L$ there exists an $f$ and an $\epsilon>0$ such that the zero set of $f$ intersected with any three-sphere of radius less than $\epsilon$ is $L$. 
Again the proof is not constructive in the sense that it would allow for an algorithm that explicitly determines one such polynomial for a given link.

The construction in Section \ref{sec:fibration} generated, for every $a,b>0$, a family of polynomials $f_{\lambda}$ in $u$, $v$ and $\overline{v}$. 
Since $\lambda$ is only a stretching factor, it is easy to see that these functions satisfy the relation 
\begin{equation}
	f_{\lambda}(u,v)=\lambda^{s}f_{1}\left(u/\lambda,v\right). 
	\label{eq:stretch}
\end{equation}

Not all polynomials $f_{\lambda}$ a priori have the property that their zero set on $S^3$ is the desired link $L$. 
Theorem \ref{isotopy} implies that for all $a_1$, $b_1$ and small enough values of $\lambda$, the constructed semiholomorphic function $f_{\lambda}:\mathbb{C}\times\mathbb{R}^{2}\longrightarrow\mathbb{C}$ has the desired knot as its nodal set on $S^{3}$ which has radius 1. 
However, the construction does not provide any information on the topology of the zero sets on spheres of different radii. 
The function does not even need to have a singularity at $0\in\mathbb{R}^{4}$ and therefore is not necessarily one of the polynomials whose existence was shown by Akbulut and King.

However, we can explicitly construct functions like this.
Consider the function 
\begin{equation}
F(u,v)=\rho^{\deg( f_{\rho})}f_{\rho}\left(\frac{u}{\rho},\frac{v}{\rho}\right)=\rho^{s+\deg( f_{\rho})}f_{1}\left(\frac{u}{\rho^{2}},\frac{v}{\rho}\right),
\label{eq:scale}
\end{equation}
where $\rho=\sqrt{|u|^{2}+|v|^{2}}$.
If $r$ is even, all exponents of $v$ and $\overline{v}$ in $f_{1}$ are even and hence, potentially after multiplying by an appropriate power of $\rho$, $F$ is a polynomial in four real variables, say $x_{1}=\mathrm{Re}(u)$, $x_{2}=\mathrm{Im}(u)$, $x_{3}=\mathrm{Re}(v)$ and $x_{4}=\mathrm{Im}(v)$.
Consider now the zero set of $F$ on the three-sphere of radius $\rho$. 
It is exactly $\rho L$, where $L$ is the zero set of $f_{\rho}$ on the unit three-sphere. 
We know that for small values of $\rho$ this is equal to the desired link by construction. 
Hence intersecting the zero set of $F$ with a three-sphere of radius $\rho$ results in the desired link for $\rho<\epsilon$. 
Note that this is the same $\epsilon$ as above.

By construction $0$ is a singular point of $F$. Also we have seen that all points on $f_{\lambda}^{-1}(0)\cap S^{3}$ are regular for small enough $\lambda$. 
Thus for small enough $\rho$, all points on $F^{-1}(0)\cap S_{\rho}^{3}$ are regular and hence the singular point at $0$ is weakly isolated.
Thus $F$ satisfies all conditions from Akbulut and King's existence theorem.  
Similar, but more elaborate arguments justify an explicit construction of real polynomials with weakly isolated singularities for any lemniscate link.

\section{Discussion}\label{sec:discussion}

We have shown how a natural generalisation of Brauner's almost century-old construction of functions for torus knots can be adapted to make polynomial maps from $\mathbb{R}^3$ (or $S^3$) to $\mathbb{C}$ whose zeros have the form of knots and links in a family we call the lemniscate knots (generalising the torus knots).
The explicit construction parametrises a braid representation of the knot trigonometrically, which become polynomials under the mathematical procedure of closing the braid.
Lemniscate knots are fibred, and with appropriate choice of the parameters, these maps are also fibrations, and for certain lemniscate knots, the maps can be extended to weakly isolated singularities.
We have outlined some physical applications of these explicit maps, and some generalisations, in linear and nonlinear physics.

The lemniscate knots form a particularly symmetric subset of spiral knots, and we have proved and conjectured properties of both families. 
Other trends may be general; for instance, on the basis of the explicit examples we have, is seems plausible that minimum braid words of lemniscate knots are (anti)palindromic.
Preliminary investigation suggests that hyperbolic volume of families of lemniscate knots for increasing $s,r$ and $\ell$, tends to increase, suggesting that other simple knots do not occur as complex lemniscate knots.

Polynomial maps $f$ can obviously be constructed for knots and links which are are closures of more general $(k,\ell)$ Lissajous braids with coprime $k,\ell \ge 2$.
Such braids are not homogeneous so such knots are not spiral, despite some properties from Theorem \ref{spiral} applying, such as being being periodic with period $r$. 
The simplest nontrivial knot is based on the $(2,3)$ Lissajous figure, with $s=5, r=2$, represented by $(\sigma_1\sigma_3\sigma_2^{-1}\sigma_4^{-1}\sigma_1^{-1}\sigma_3^{-1}\sigma_2\sigma_4)^2$.
This apparently closes to a knot with $16$ crossings, so these knots seem likely to appear beyond standard tabulations.
Similar knots formed as closures of general Lissajous braids, without the restriction that $s$ is less than $k$ or $\ell$, have been discussed in \cite{sv:2016:lissajous}.

Thus the construction of maps $f$ extends beyond the lemniscate knots which have been our emphasis, not only to higher Lissajous figures, but also the generalisations in Section \ref{sec:fibration} which included a composite and a cable knot.
All of these generalisations apply to the (trigonometric) braid, but the same basic algebraic procedure leads to the polynomial map $f$.
In particular, the identification of the braid with the zeros of the polynomial family $p_h(u)$ from $\mathbb{C}\times S^1 \longrightarrow \mathbb{C}$ guarantees that $f$ is semiholomorphic, which is a powerful property (it is this property that allows the degree of the rational map described in Section \ref{sec:applications} to be controlled).

Having explicit functions for knotted fields is particularly useful in applications to construct knotted fields in physical systems.
The polynomial maps $f$ and $F$ which can be adapted for a variety of systems (either simply by multiplying by an overall weight function to aid normalisation, or in a more complicated way to alter the domain and target spaces).
The procedure in fact gives polynomials whose coefficients are all integers; since the knotted zeros are structurally stable, the coefficients may be smoothly perturbed (such as the holograms in \cite{dkjop:2010isolated}, or the rational map (\ref{eq:fig8hopf})) to optimise some physically-motivated merit function without destroying the nodal topology.   
From a physical point of view, the mathematical procedure is a way of finding a topologically-interesting region of (possibly high-dimensional) parameter space, for which physical principles are required to understand the significance of knotted fields.

\appendix
\section{Appendix}

\subsection{Proof of Theorem \ref{rob}}
\begin{proof}
We begin by showing the braid word of the $(s=2n+1,r=2,\ell=2)$ lemniscate knot can be written $\sigma_1^{-n} \sigma_2 \sigma_1^{-1} \sigma_2^{n}$.
This then agrees (for low values of $n$) with the minimum braid word for the knots in Table \ref{tab:ell=2} found by \cite{gittings:2004minimum}.
The basic word for $(s=2n+1,\ell=2)$ lemniscate braids with $b>0$ is, from the discussion in Section \ref{sec:braids}, given by $w^{(2n+1,2)} = \sigma_1^{-1} \sigma_2^{-1} \cdots \sigma_n^{-1} \sigma_{n+1} \cdots \sigma_{2n}$ (since anagrams of isograms are equivalent).
The knots in question thus have braid word
\begin{eqnarray*}
   [w^{(2n+1,2)}]^2 & = & \sigma_1^{-1} \sigma_2^{-1} \cdots \sigma_n^{-1} \sigma_{n+1} \cdots \sigma_{2n} \sigma_1^{-1} \sigma_2^{-1} \cdots \sigma_n^{-1} \sigma_{n+1} \cdots \sigma_{2n} \\
   & = & \sigma_1^{-1} \sigma_2^{-1} \sigma_1^{-1} \sigma_3^{-1} \sigma_2^{-1} \sigma_4^{-1} \cdots \sigma_{n-1}^{-1} \sigma_n^{-1} \sigma_{n+1} \sigma_n^{-1} \cdots \sigma_{2n} \sigma_{2n-1} \sigma_{2n}. 
\end{eqnarray*}
The rearrangement in the second line represents a rearrangement where all generators are moved as near to the beginning of the word without meeting another generator with which they do not commute.
Starting now from the left of this word, using the braid group property and Markov stabilisation, we have 
\[ \sigma_1^{-1} \sigma_2^{-1} \sigma_1^{-1} \sigma_3^{-1} \cdots = \sigma_2^{-1} \sigma_1^{-1} \sigma_2^{-1} \sigma_3^{-1} \cdots = \sigma_2^{-2} \sigma_3^{-1} \cdots.\]
This procedure is applied again, with application of the braid group relation needed three times,
\[ \sigma_2^{-2} \sigma_3^{-1} \sigma_2^{-1} \sigma_4^{-1} \cdots = \sigma_2^{-1} \sigma_3^{-1} \sigma_2^{-2} \sigma_4^{-1} \cdots = \sigma_3^{-1} \sigma_2^{-1} \sigma_3^{-1} \sigma_2^{-1} \sigma_4^{-1} \cdots 
=  \sigma_3^{-2} \sigma_2^{-1} \sigma_3^{-1} \sigma_4^{-1} \cdots = \sigma_3^{-3} \sigma_4^{-1} \cdots. \]
Repeating, we get the word
\[ [w^{(2n+1,2)}]^2 = \sigma_n^{-n} \sigma_{n+1} \sigma_n^{-1} \cdots \sigma_{2n} \sigma_{2n-1} \sigma_{2n}. \]
Following the same procedure from the right ultimately results in the word
\[ [w^{(2n+1,2)}]^2 = \sigma_n^{-n} \sigma_{n+1} \sigma_n^{-1} \sigma_{n+1}^{n}. \]
As all other strands have been eliminated, the result follows on relabelling the generators.

To prove the form of the Alexander polynomial, we use the Burau representation of the braid group \cite{kt:2008braid} on the braid word $\sigma_{1}^{n}\sigma_{2}^{-1}\sigma_{1}\sigma_{2}^{-n}$. Since the Alexander polynomial does not distinguish between mirror image (and the knots in question are ampicheiral anyway), the calculated Alexander polynomial is that of $L(s=2n+1,r=2,\ell=2)$.
The matrix representatives of $\sigma_1$ and $\sigma_2^{-1}$ are given respectively by
\[ \Sigma_1 = \left( \begin{array}{ccc} 1 - t & t & 0 \\ 1 & 0 & 0 \\ 0 & 0 & 1 \end{array}\right), \qquad 
 \Sigma_2^{-1} = \left( \begin{array}{ccc} 1 & 0 & 0 \\ 0 & 0 & 1 \\ 0 & t^{-1} & 1 - t^{-1} \end{array}\right), \]
of which the $n$th powers are (as easily verifiable by induction)
\[ \Sigma_1^n = \left( \begin{array}{ccc} A_n & t A_{n-1} & 0 \\ A_{n-1} & t A_{n-2} & 0 \\ 0 & 0 & 1 \end{array}\right), \qquad 
 \Sigma_2^{-n} = \left( \begin{array}{ccc} 1 & 0 & 0 \\ 0 & t^{-1}B_{n-2} & B_{n-1} \\ 0 & t^{-1} B_{n-1} & B_n \end{array}\right), \]
where $A_n = \sum_{j=0}^n (-t)^j$ and $B_n = \sum_{j=0}^n (-t^{-1})^j$.
This means the full braid word is represented by
\[ \Sigma_1^n \Sigma_2^{-1} \Sigma_1 \Sigma_2^{-n} = \left( \begin{array}{ccc} (1-t) A_n & A_{n-1} B_{n-1} + A_n B_{n-2} & t(A_{n} B_{n-1} + A_{n-1} B_n) \\
(1-t) A_{n-1} & A_{n-1} B_{n-2} + A_{n-2}B_{n-1} & t(A_{n-1}B_{n-1} + A_{n-2}B_n) \\
t^{-1} & t^{-1}(1-t^{-1}) B_{n-1} & (1-t^{-1}) B_n \end{array} \right). \]
The Alexander polynomial $\Delta_L(t)$ is now found by computing the determinant of a $2\times 2$ minor of this matrix minus the identity matrix, from which it can be seen (using $A_{n-1} = t^{-1}(1-A_n), B_{n-1} = t(1-B_n)$) that $\Delta_L(t) = 1 - A_n - B_n + 2 A_n B_n$.
The result follows directly from the fact that $A_n B_n = \sum_{j,k=0}^n (-t)^j (-t^{-1})^k = \sum_{m = -n}^n (n+1-|m|)(-t)^m$. \end{proof}

\subsection{Proof of Theorem \ref{conway}}
\begin{proof}
Let $L$ be a spiral knot with $r=2$. 
Then it is the closure of a braid word of the form $w^2$, where $w=\sigma_{1}^{\varepsilon_{1}}\sigma_{2}^{\varepsilon_{2}}\ldots\sigma_{s-1}^{\varepsilon_{s-1}}$.  
We write the vector $\boldsymbol{\varepsilon}=(\varepsilon_{1},\varepsilon_{2},\ldots,\varepsilon_{s-1})$ as $(\varepsilon_{1,1},\varepsilon_{1,2},\ldots,\varepsilon_{1,n_{1}},\varepsilon_{2,1},\ldots,\varepsilon_{2,n_{2}},\ldots,\varepsilon_{\ell,n_{\ell}})$ with $\varepsilon_{i,j}=\varepsilon_{i,k}$ for all $j, k\in\{1,\ldots,n_{i}\}$ and $\varepsilon_{i,n_{i}}=-\varepsilon_{i+1,1}$.

Figure \ref{fig:diagram} (b) shows a diagram of the closed braid. 
Any closing arc connecting the bottom end of one strand with the top end of another can be taken, and placed either on top or below the braid as shown in Figure \ref{fig:diagram} (c), which can be done so that it cancels all crossings on the top.
The resulting diagram can be easily brought into the form of a 4-plait as in figure \ref{fig:diagram}(d), so $L$ is rational.
Note that by placing the closing arcs such that they cancel the crossings on the top, they also cancel the crossings on the bottom if $\varepsilon_{i}=\varepsilon_{i+1}$. 
This gives a 4-plait (Figures \ref{fig:diagram} (d), (e)) with Conway notation $[\varepsilon_{1,1}(n_{1}+1),\ 2\varepsilon_{2,1},\ \varepsilon_{2,1}n_{2},\ 2\varepsilon_{3,1},\ \ldots\ \varepsilon_{\ell-1,1}n_{\ell-1},\ 2\varepsilon_{\ell,1},\ \varepsilon_{\ell,1}n_{\ell}]$. 
Note that the tangle sign convention means that the $i$th entry in the Conway notation if $(-1)^{i}$ times the number in the $ith$ box from the left in the diagram.

\begin{figure}
\begin{center}
\includegraphics[width=\textwidth]{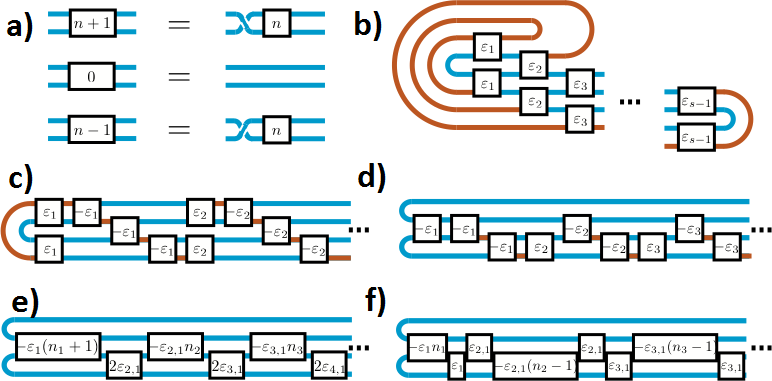}
\caption{
   Sequence of diagrams from the closed braid to a minimal diagram in 2-bridge form. 
   (a) Definition of crossing notation.
   (b) Knot diagram of the spiral braid closure. 
   Arcs which are part of the braid are blue and closing arcs are orange.
   (c) Placing orange arcs either above or below the braid diagram, gives a new diagram.
   (d) With the right choice of `above' or `below', all crossings in the top row cancel, resulting in the diagram of a closed tangle. 
   The bottom left crossing in (c) (with sign $\varepsilon_{1}$) has been moved into the middle row (now with sign $-\varepsilon_{1}$).
   (e) The knot diagram that results from canceling the crossings in the bottom row when possible.
   (f) Applying Cromwell's move allows neighbouring boxes with entries $a$ and $b$ to be transferred to boxes with entries $a-1$, $-1$ and $b-1$ or $a+1$, $1$ and $b+1$. 
   A sequence of these moves leads to the depicted diagram, which is reduced and alternating and hence minimal.}
\label{fig:diagram}
\end{center}
\end{figure}

For the following it is convenient to work with a different notation, where we define $\{b_{1},\ b_{2},\ b_{3},\ \ldots \ b_{n}\}=[-a_{1},\ a_{2},\ -a_{3},\ \ldots\ (-1)^{n}a_{n}]$ using the usual Conway notation. 
With this notation the numbers $b_{i}$ exactly correspond to the numbers in the boxes in \ref{fig:diagram}.
In his book \cite{cromwell:2004knots}, Cromwell describes a move that shows that a link with notation $\{b_{1},\ \ldots\ b_{i},\ b_{i+1}\ \ldots\}$ is equivalent to both $\{b_{1},\ \ldots b_{i}-1,\ -1,\ b_{i+1}-1\ \ldots\}$ and $\{b_{1},\ \ldots b_{i}+1,\ 1,\ b_{i+1}+1\ \ldots\}$, which shows that \[\{-\varepsilon_{1,1}(n_{1}+1),\ 2\varepsilon_{2,1},\ -\varepsilon_{2,1}n_{2},\ 2\varepsilon_{3,1},\ \ldots\ -\varepsilon_{\ell-1,1}n_{\ell-1},\ 2\varepsilon_{\ell,1},\ -\varepsilon_{\ell,1}n_{\ell}\}\]
 is equivalent to 
\begin{align*}&\{-\varepsilon_{1,1}n_{1},\ \varepsilon_{1,1},\ \varepsilon_{2,1},\ -\varepsilon_{2,1}(n_{2}-1),\ \varepsilon_{2,1},\ \varepsilon_{3,1},\ -\varepsilon_{3,1}(n_{3}-1),\ \varepsilon_{3,1},\ \varepsilon_{4,1},\ldots\\
 &\ \ldots\ -\varepsilon_{\ell-1,1}(n_{\ell-1}-1),\  \varepsilon_{\ell-1,1},\ \varepsilon_{\ell,1},\ -\varepsilon_{\ell,1}n_{\ell}\}.\end{align*}
Going back to the usual tangle notation, from the definition of $\boldsymbol{\varepsilon}$ in (\ref{eq:epscondition}), this is equal to the desired Conway notation and the corresponding diagram has crossing number $s+\ell-1$.
It is easy to check that the diagram is reduced and alternating and thus by the proof of one of Tait's conjectures by Kauffman \cite{kauffman:1987state}, Thistlewaite \cite{thistlethwaite:1987spanning,thistlethwaite:1988kauffman} and Murasugi \cite{murasugi:1987i,murasugi:1987ii} the diagram is minimal. 
Hence we have $c(L)=s+\ell-1$.
Since $L$ is alternating and fibred by Theorem \ref{spiral}\ref{num:fibred}, it follows from \cite{murasugi:1991braid} that $b_{\mathrm{ind}}(L)=s+\ell-1-(s-1)+1=\ell+1$ if $L$ is a knot.
We believe the statements could alternatively have been proven using Theorem A and Theorem B from \cite{murasugi:1991braid} as well. \end{proof}

\begin{acknowledgements}
The authors are grateful for discussions over several years with many colleagues on this subject, particularly Gareth Alexander, Mark Bell, Michael Berry, David Chillingworth, William Irvine, Randy Kamien, Lou Kauffman, Kevin O'Holleran, Miles Padgett, Daniel Peralta-Salas, Jonathan Robbins, Radmila Sazdanovic, Paul Sutcliffe and Sandy Taylor.  
Preliminary calculations on higher Lissajous braids were made by Jonathan Holden on a Nuffield Research Placement, and on hyperbolic volumes of lemniscate knots by Sandy Taylor.
MRD is grateful to the KITP for hospitality during some of this work.
BB and DF are supported by the Leverhulme Programme Grant `Scientific Properties of Complex Knots'. 
RPK was supported by the Leverhulme Research Grant `Knots of Light in Nature'.   
MRD acknowledges support from a Royal Society University Research Fellowship during part of this work.
\end{acknowledgements}



\end{document}